\documentclass[12pt]{amsart}
\usepackage[utf8]{inputenc}
\usepackage{enumerate,enumitem,amsmath,amssymb,mathrsfs,stmaryrd}
\usepackage{latexsym,epsf,graphicx,comment,appendix}

\newfont{\bb}{msbm10 at 12pt}
\newfont{\tbb}{msbm10 at 8pt}

\def\r{\hbox{\bb R}}

\def\h{\hbox{\bb H}}

\newcommand{\meta}[2]{\langle #1,#2 \rangle }

\baselineskip 
\usepackage {amsmath}
\usepackage {amsthm}
\usepackage{times}
\usepackage{amscd}
\usepackage{epsf}

\numberwithin{equation} {section}

\begin{document}
\mbox{}\vspace{0.2cm}\mbox{}

\theoremstyle{plain}\newtheorem{lemma}{Lemma}[section]
\theoremstyle{plain}\newtheorem{proposition}{Proposition}[section]
\theoremstyle{plain}\newtheorem{theorem}{Theorem}[section]

\theoremstyle{plain}\newtheorem*{theorem*}{Theorem}
\theoremstyle{plain}\newtheorem*{main theorem}{Main Theorem} 
\theoremstyle{plain}\newtheorem*{lemma*}{Lemma}
\theoremstyle{plain}\newtheorem*{claim}{Claim}

\theoremstyle{plain}\newtheorem{example}{Example}[section]
\theoremstyle{plain}\newtheorem{remark}{Remark}[section]
\theoremstyle{plain}\newtheorem{corollary}{Corollary}[section]
\theoremstyle{plain}\newtheorem*{corollary-A}{Corollary}
\theoremstyle{plain}\newtheorem{definition}{Definition}[section]
\theoremstyle{plain}\newtheorem{acknowledge}{Acknowledgment}
\theoremstyle{plain}\newtheorem{conjecture}{Conjecture}

\begin{center}
\rule{15cm}{1.5pt} \vspace{.4cm}

{\bf\Large On $n$-superharmonic functions and some geometric applications} 
\vskip .3cm

Shiguang Ma$\, ^\dag$\footnote{The author is
supported by NSFC 11571185 and NSFC 11871283} 
and \hskip 0.15cm Jie Qing$\,^\ddag$\footnote{The author is partially supported by NSF DMS-1608782}\\

\vspace{0.3cm} 
\rule{15cm}{1.5pt}
\end{center}

\vskip 0.3cm
\noindent$\mbox{}^\dag$ Department of Mathematics, Nankai University, Tianjin, China; \\e-mail: 
msgdyx8741@nankai.edu.cn 
\vspace{0.2cm}

\noindent $\mbox{}^\ddag$ Department of Mathematics, University of California, Santa Cruz, CA 95064; \\
e-mail: qing@ucsc.edu

\title{}
\begin{abstract} In this paper we study asymptotic behavior of $n$-superharmonic functions at isolated singularity
using the Wolff potential and $n$-capacity estimates in nonlinear potential theory. Our results are inspired by and 
extend \cite{AH} of Arsove-Huber and \cite{Tal-2} of Taliaferro in 2 dimensions. To study 
$n$-superharmonic functions we use a new notion of thinness by $n$-capacity motivated by a type of 
Wiener criterion in \cite{AH}. To extend \cite{Tal-2}, we employ the Adams-Moser-Trudinger 
inequality for the Wolff potential, which is inspired by the one used in \cite{BM} of Brezis-Merle (cf. \cite{Io}). 
For geometric applications, we study the asymptotic end behavior of complete
conformally flat manifolds as well as complete properly embedded hypersurfaces in hyperbolic space, both with nonnegative Ricci curvature. 
These geometric applications seem to elevate the importance of $n$-Laplace equations and make 
a closer tie to the classic analysis developed in conformal geometry in general dimensions.
\end{abstract}

\maketitle

%%%%%%%%%%%%%%%%%%%%%%%%%%%%%%%%%%%%%%%%%%%%%%%%%%%%%%%%%%%%%%%%%%%%%%%%%%%%%%%%%%%%%%%%%%%%%%

\section{Introduction}\label{Sec:Intro}

In this paper we will develop some understanding of isolated singularities of $n$-superharmonic functions in $n$ dimensions and apply it
to studying some geometric problems. Recall that the $n$-Laplace operator
\begin{equation}
\Delta_n u = \text{div}(|\nabla u|^{n-2}\nabla u)
\end{equation}
is a quasilinear, possibly degenerate, elliptic operator that agrees with the Laplace operator 
in $2$ dimensions. 
\\

The theory of n-Laplace equations is as fundamental as that of classic Laplace 
equations since it is also in the center of the interplay of several important fields of mathematics including calculus of 
variations, partial differential equations, nonlinear potential theory, and mathematical physics. Obviously the theory of 
$n$-Laplace equations is more interesting as well as more challenging, because the principle of superposition is no 
longer available, instead, understanding of interactions is indispensable. We would like to develop higher dimensional extensions to 
what have been done for the theory of subharmonic functions in \cite{AH, HK, Tal-2} (references therein) regarding 
asymptotic behavior and their applications in differential geometry. Our research in this paper seems to elevate the importance of 
$n$-Laplace equations and makes a closer tie to the classic analysis developed in conformal geometry.
\\

The first goal for us is to study the behavior of $n$-superharmonic functions at a point. The first main theorem
in general dimensions is inspired by and extends the work of Arsove-Huber in \cite[Theorem 1.3]{AH}.

\begin{theorem}\label{Thm:main theorem-1-intro}
Let $w$ 
be a nonnegative lower semi-continuous function that is $n$-superharmonic in $B(0, 2)\subset \r^n$ and 
\[
-\Delta_n w=\mu\ge 0
\]
for a Radon measure $\mu\ge 0$.
Then there is a set $E\subset\mathbb{R}^n$, which is $n$-thin at the origin, such that
\[
\lim_{x\notin E \text{ and } |x|\rightarrow 0}\frac{w(x)}{\log\frac{1}{|x|}}=\liminf_{|x|\rightarrow0}\frac{w(x)}{\log\frac{1}{|x|}}=m\ge 0
\]
and
\[
w(x)\geq m\log\frac{1}{|x|} - C \ \text{ for $x\in B(0, 1)\setminus \{0\}$ and some $C$}.
\]
Moreover, if $w\in C^2(B(0,2)\setminus \{0\})$ and $(B(0, 2)\setminus \{0\},e^{2w}|dx|^2)$ is complete at the origin, then $m\geq1$. 
\end{theorem}

The definition of $n$-superharmonic functions is given in Definition \ref{Def:superh}. 
The definition of $n$-thinness is given in Definition \ref{Def:n-thin}, which is inspired by and extends 
the definition of thinness in \cite{AH} (see the discussion on the comparison of different notions of thinness in Section 
\ref{Subsect:2 dimensions revisit}).  The proof of Theorem \ref{Thm:main theorem-1-intro} combines the 
blow-down argument from \cite{KV} and the nonlinear potential theory \cite{AM, HKM, KM, Lind, PV} for $n$-Laplace equations, particularly
the use of the Wolff potential and $n$-capacity estimates.
\\
 
The proof of Theorem \ref{Thm:main theorem-1-intro} consists of four major steps. The first is to 
use nonlinear potential theory, particularly \cite[Theorem 1.6 and Lemma 3.9]{KM} on the Wolff potential and $n$-capacity estimates
(cf. \eqref{Equ:potential-estimate}) to show that, the blow-down quotient 
\begin{equation}
w_r(\xi) = \frac {w(r\xi)}{\log\frac 1{|r|}}
\end{equation} 
is bounded outside a subset $\hat E$ that is $n$-thin. The second step is to use a 
cut-off technique from \cite{DHM} to modify and cut off the unbounded part in order to take sequential limit for the 
blow-down quotients as $r\to 0$. Based on Liouville Theorem of \cite{Serrin-1, Red, HKM}, one knows that the sequential limits 
are all constants. 
In the third step, we use comparison principle (cf. \cite[Lemma 3.1]{Tol} and \cite{KV, KV-e}) to conclude that all sequential limits 
have to be the same as $m$ in the Theorem \ref{Thm:main theorem-1-intro}.  In the final step, based on the uniqueness of 
sequential limits, we re-run the proof in the first step to extract a subset $E$ that is $n$-thin and finish the proof of
Theorem \ref{Thm:main theorem-1-intro}. The first run of the argument in the first step is to get bounds; while the second run 
is to get uniform convergences. It is essential and very interesting to see how the classic Paul
du Bois-Reymond Theorem (cf. \cite{Rey} and \cite[(5) Page 40]{Brom}) 
for infinite series helps to re-enforce the argument in the first run in applying
\cite[Theorem 1.6 and Lemma 3.9]{KM} (cf. \eqref{Equ:improved-potential-estimate}) to get the uniform convergence.
\\

Our second goal is to eliminate nontrivial $n$-thin subsets $E$ in Theorem \ref{Thm:main theorem-1-intro}. This theorem in general
dimensions is inspired by and extends the work of Taliaferro in \cite{Tal, Tal-1, Tal-2}. 

\begin{theorem}\label{Thm:main theorem-2-intro}
Let $w \in C^2(B(0, 2)\setminus \{0\})$ be nonnegative and satisfy
\begin{equation}
 -\Delta_n w = f(x, w, \nabla w)
\end{equation}
in $B(0, 2)\setminus\{0\} \subset\mathbb{R}^n$ and that
$$
\lim_{x\to 0} w(x) = +\infty, 
$$ 
where $f$ satisfies the critical growth condition
\begin{equation}
0 \leq f (x, w, \nabla w) \leq C|\nabla w|^{n-2}e^{2w}
\end{equation}
for some fixed constant $C$. Then 
\begin{equation}
\lim_{x \rightarrow 0}\frac{w(x)}{\log\frac{1}{|x|}}= m \geq 0
\end{equation}
and
\begin{equation}
w(x)\geq m\log\frac{1}{|x|} - L \ \text{ for $x\in B(0, 1)\setminus \{0\}$}
\end{equation}
for some constant $L$. 
\end{theorem}

The essential ingredient in the proof of Theorem \ref{Thm:main theorem-2-intro} is a type of Brezis-Merle inequality 
\begin{equation}\label{Equ:b-m}
\int_{\Omega} {\rm exp}( {\frac{n(1-\delta)W_{1,n}^{\mu_{f}}(x,D)}{\|f\|_{L^{1}(\Omega)}^\frac 1{n-1}}})dx\leq 
\frac {c(n)2^{2n+1} |B(0, D)|}{ \delta^{n+1}} + 2^n |\Omega|. 
\end{equation}
for the Wolff potential
$$
W_{1, n}^{\mu_f}(x, D) = \int_0^D \mu_f(B(x, s))^\frac 1{n-1}\frac {ds}s
$$
with 
$$
\mu_f (A) = \int_{A\cap\Omega} f(x)dx
$$
induced by a nonnegative function in $L^1(\Omega)$, where $D$ is the diameter of $\Omega$ and $\delta\in (0, 1)$. \eqref{Equ:b-m}
is stated in Proposition \ref{Pro:brezis-merle} and 
extends the one discovered by Brezis-Merle in 2 dimensions  (cf. \cite{BM, FM, Io} and references therein). 
This Adams-Moser-Trudinger inequality \eqref{Equ:b-m} for the Wolff potential 
helps control any possible concentration and rule out any 
possible nontrivial $n$-thin subset $E$ in Theorem \ref{Thm:main theorem-1-intro}. As stated in Remark \ref{Rem:critical growth},
the critical growth condition may be described as
\begin{equation}\label{Equ:critical growth-intro}
0 \leq f(x, w, \nabla w) \leq C |\nabla w|^p e^{\alpha w}
\end{equation}
for any $0 < p < n$ and $\alpha > 0$ to be more general. Our proof of Theorem \ref{Thm:main theorem-2-intro} is a streamlined one
from \cite{Tal, Tal-1, Tal-2} with the help of the Adams-Moser-Trudinger inequality \eqref{Equ:b-m} for the Wolff potential.
\\

As applications we first want to study asymptotic behavior at the end of complete locally conformally flat manifolds with nonnegative 
Ricci. After the classification theorems of \cite{Zhu, CH-1}, we want to focus on complete metrics $e^{2\phi}|dx|^2$ on $\r^n$. One may 
calculate and find that
\begin{equation}\label{Equ:extends gauss equation}
- \Delta_n \phi = \text{Ric}_g(\nabla^g \phi) |\nabla\phi|^{n-2}e^{2\phi},
\end{equation}
where $\text{Ric}_g(\nabla\phi)$ is the Ricci curvature of the metric $g= e^{2\phi}|dx|^2$ in $\nabla^g \phi$ direction. 
\eqref{Equ:extends gauss equation} is clearly a generalization of Gauss curvature equations in higher dimensions.

\begin{theorem}\label{Thm:main theorem-3-intro}
Suppose that $(\mathbb{R}^{n},e^{2\phi}|dx|^{2})$
is complete with nonnegative Ricci, where $\phi$ is a smooth function. Then there is a subset
$E\subset\mathbb{R}^{n}$, which is $n$-thin
at infinity, such that 
\begin{equation}
\lim_{x\notin E\text{ and } x \to\infty}\frac{\phi(x)}{\log\frac{1}{|x|}}=\liminf_{x\to\infty}\frac{\phi(x)}{\log\frac{1}{|x|}}=m
\end{equation}
and 
\begin{equation}
\phi(x)\geq m\log\frac{1}{|x|} - L
\end{equation}
for some constant $L$, where $m\in [0, 1]$ and
\begin{equation}
m^{n-1} =\frac 1{w_{n-1}} \int_{\mathbb{R}^{n}}\text{Ric}_{g}(\nabla^{g}\phi)|\nabla\phi|^{n-2}e^{2\phi}dx.
\end{equation}
Moreover, 
\begin{itemize}
\item $m=0$ if and only if $g = e^{2\phi}|dx|^2$ is flat;
\item if $\text{Ric}_{g}$ is bounded in addition, then 
\begin{equation}
\lim_{x\to\infty}\frac{\phi(x)}{\log\frac{1}{|x|}}=\liminf_{x\to\infty}\frac{\phi(x)}{\log\frac{1}{|x|}}=m.
\end{equation}
\end{itemize}
\end{theorem}

This theorem gives some precise description of asymptotic end behavior. More importantly it also includes a rigidity result that does 
not assume Ricci is bounded. The rigidity result in this theorem should be compared with \cite[Theorem 0.3]{Cold} and 
\cite{BKN, CZ, CH-1}. It is particularly desirable to compare the blow-down approaches here and those 
in \cite{BKN, Cold, CH-1}. The proof appeals to
Theorem \ref{Thm:main theorem-1-intro} and Theorem \ref{Thm:main theorem-2-intro}. But it is not straightforward at all to calculate $m$ and
derive the rigidity, especially when Ricci is not assumed to be bounded. Our argument relies on the ingenious construction of exhausting 
family of domains to perform integrations (please see $\Omega_{\varepsilon, t}^\pm$ in the proof of Theorem \ref{Thm:main theorem-g}
in Section \ref{Sec:conformally flat}). 
\\

Our second application is to study the asymptotic behavior at the end of properly embedded complete hypersurfaces with nonnegative
Ricci curvature in hyperbolic space. It was shown in \cite[Main Theorem]{BMQ-r} that such hypersurfaces have at most two ends, 
and are equidistant hypersurfaces if with two ends. Based on Theorem \ref{Thm:main theorem-1-intro}, we are able to improve the 
theorems on asymptotic at infinity in \cite{AlCu, AlCu2} assuming only Ricci to be nonnegative.

\begin{theorem}\label{Thm:main theorem-4-intro}
Suppose that $\Sigma^n$ is a properly embedded, complete hypersurface with nonnegative Ricci curvature and one single end in
hyperbolic space $\h^{n+1}$. Then it is a global graph of $\rho= \rho(x)$ in Busemann coordinates and it is asymptotically rotationally 
symmetric in the sense that there is a number $m\in[0,1]$ such that 
\begin{equation}
m\log|x|+o(\log|x|)\leq \rho(x)\leq m\log|x|+C
\end{equation}
as $x \to \infty$ in $\r^n$. Moreover, $m=0$ implies that the hypersurface is a horosphere.  In any case, the
hypersurface $\Sigma$ always stays inside a horosphere and is supported by some equidistant hypersurface. 
\end{theorem}

The proof follows from the one in \cite{AlCu, AlCu2}, and in fact is simpler than the one in \cite{AlCu, AlCu2}, because
of Theorem \ref{Thm:main theorem-1-intro}. The use of $n$-subharmonic functions is more suitable than the use 
of subharmonic functions restricted to each 2-plane (cf. \cite{AlCu, AlCu2}). To eliminate any nontrivial $n$-thin set $E$ in 
Theorem \ref{Thm:main theorem-1-intro} in these cases, we use the strict convexity of the hypersurfaces, where the strong
$n$-capacity lower bound estimate of Gehring (cf. \cite[Lemma 1.4 page 212]{Red1} and \cite[Theorem 4]{G}) is used sharply.
\\

The rest of the paper is organized as follows: In Section \ref{Sec:preli}, we present definitions and basic facts that are useful. 
We describe what have been done in 2 dimensions to motivate our study in this paper. We also explain the opportunity for the 
use of $n$-superharmonic functions in geometric problems. In Section \ref{Sec:Arsove-Huber}, we define $n$-thinness by 
$n$-capacity and prove Theorem \ref{Thm:main theorem-1-intro}. In Section \ref{Sec:taliaferro}, we establish the 
Adams-Moser-Trudinger inequality for the Wolff potential and prove Theorem \ref{Thm:main theorem-2-intro}. 
In Section \ref{Sec:conformally flat}, we introduce the classification of complete locally conformally flat manifolds with 
nonnegative Ricci curvature and prove Theorem \ref{Thm:main theorem-3-intro}.
Finally, in Section \ref{Sec:hypersurfaces}, we recall the classification of complete properly embedded hypersurfaces with 
nonnegative Ricci curvature and prove Theorem \ref{Thm:main theorem-4-intro}.
 
%%%%%%%%%%%%%%%%%%%%%%%%%%%%%%%%%%%%%%%%%%%%%%%%%%%%%%
%%%%%%%%%%%%%%%%%%%%%%%

\section{Preliminaries and background}\label{Sec:preli}

In this section,  after adopting the definitions of $n$-harmonic functions and $n$-superharmonic 
functions from \cite[Section 2]{KM}, we would like to first present a review of what have been done in 2 dimensions to 
motivate what we want to do in 
general dimensions. Then we would like to introduce some background and tools from the theory of quasilinear elliptic equations
and nonlinear potential theory that are useful to us. We also introduce the geometric problems that we expect to use 
$n$-superharmonic functions to study in this paper.

%%%%%%%%%%%%%%%%%%%%%%%

\subsection{Definitions of $n$-superharmonic functions}

We want to have a discussion on definitions of $n$-superharmonic functions first to clear any possible confusions caused by
terminology. Let us recall the definitions of $n$-superharmonic functions from \cite[Definition 2.5 and 2.12]{Lind} and \cite[Section 2]{KM}. 

\begin{definition} (\cite[Definition 2.5]{Lind}) 
For a domain $\Omega\subset\mathbb{R}^n$, a function $u\in W^{1, n}_{\text{loc}}(\Omega)$ is said to be weakly $n$-harmonic 
in $\Omega$ if
$$
\int |\nabla u|^{n-2}\nabla u\cdot \nabla\phi = 0
$$
for all $\phi\in C^\infty_c(\Omega)$. A weak $n$-harmonic function $u\in W_{\text{loc}}^{1, n}(\Omega)$
is said to be $n$-harmonic if it is continuous in $\Omega$.
\end{definition}

We know from \cite[Theorem 2.19]{Lind} that any weak $n$-harmonic function  always has a continuous representive and therefore $n$-harmonic. 
For further regularity of $n$-harmonic functions we refer readers to \cite{Lind} and references therein.
For the definitions of $n$-superharmonic functions, we first recall

\begin{definition} \label{Def:weak-superh} (\cite[Definition 2.12]{Lind}) For a domain $\Omega\subset \mathbb{R}^n$ and a function
$u \in W_{loc}^{1,n}(\Omega) $ satisfying 
\begin{equation}
\int \meta {|\nabla u|^{n-2}\nabla u}{\nabla \eta} dx \geq 0 \quad \text{for each} \,\,\eta \in C_{0}^{\infty}(\Omega) \ \text{and $\eta \geq 0$}
\end{equation}
is called a weak supersolution to $n$-Laplace equation in $\Omega$. A function $u$ is called a weak subsolution if $-u$ is a weak
supersolution.
\end{definition}

In the mean time, the following definition for $n$-superharmonic functions is often used in nonlinear potential theory. To avoid confusions,
we quote the following definition for $n$-superharmonic functions.

\begin{definition}\label{Def:superh} (\cite[Section 2]{KM}) 
For a domain $\Omega\subset\mathbb{R}^n$ and a lower semi-continuous function 
$$
u: \Omega \to (-\infty, +\infty]
$$
is said to be $n$-superharmonic in $\Omega$ if $u$ is not identically infinite in each connected component of $\Omega$, 
and if for all bounded open set $D\subset \bar D\subset \Omega$ and all $h\in C(\bar D)$ that is $n$-harmonic in $D$, $h \leq u$ in
$\partial D$ implies that $h\leq u$ in $D$. A function $u$ is said to be $n$-subharmonic in $\Omega$ if $-u$ is 
$n$-superharmonic in $\Omega$.
\end{definition}

Fortunately, the relations between these two definitions has been clarified very well in \cite{HeKi, KM}. For instance, we have

\begin{lemma} \label{Lem:equiv def}(\cite[Proposition 2.7]{KM})
\begin{itemize}
\item If $u$ is a weak supersolution to $n$-Laplace equation in $\Omega\subset \mathbb{R}^n$, then there is an $n$-superharmonic function 
$v$ such that $v = u$ a.e. in $\Omega$;
\item If $u$ is $n$-superharmonic in $\Omega$ and $u\in W^{1, n}_{\text{loc}}(\Omega)$, then $u$ is weak supersolution to $n$-Laplace equation;
\item If $u$ is $n$-superharmonic and locally bounded, then $u\in W^{1, n}_{\text{loc}}(\Omega)$ and is a weak $n$-supersolution to $n$-Laplace equation.
\end{itemize}
\end{lemma}

Clearly, when functions are $C^2$ or better, these two definitions agree,  we will simply refer them $n$-superharmonic with no confusion.
For $n$-superharmonic functions, one still has integrability of the gradient as shown in \cite[Theorem 5.15]{Lind}.

\begin{lemma}\label{Lem:nabla-q} (\cite[Theorem 5.15]{Lind}) 
Suppose that $u$ is an $n$-superharmonic function in $\Omega$. Let $D\subset \subset\bar D\subset\Omega$ be a bounded
subdomain and $0 < q < n$. Then there is a constant $C>0$ such that
$$
\int_D |\nabla u|^q dx \leq C.
$$
\end{lemma}
Therefore, if $u$ is $n$-superharmonic or weakly $n$-superharmonic function, then 
$\mu = - \Delta_n u $ may be considered to be a nonnegative Radon measure on $\Omega$ (cf. \cite{Lind} and \cite[Theorem 2.1]{KM-92}). 
And, by a simple approximation argument,  
\begin{equation}\label{Equ:integral-radon}
\int |\nabla u|^{n-2}\nabla u\cdot\nabla \phi = \int \phi d\mu
\end{equation}
for any testing function $\phi \in W^{1, n}_0(D)$, if $u\in W^{1, n}(D)$ and $D\subset \Omega$. 
It is also helpful to mention the following weak comparison principle from Theorem 2.15 and the remark right after the proof 
in \cite{Lind}.

\begin{theorem}(\cite[Theorem 2.15]{Lind}) \label{Lem:WeakCompPrinc}
Suppose that $u$ is a weakly $n$-superharmonic function and $v$ is an $n$-harmonic function in a bounded domain 
$\Omega \subset \mathbb{R}^n$. If for every $\zeta \in \partial \Omega$
\begin{equation}
\limsup_{x\rightarrow\zeta} v(x)\leq \liminf_{x\rightarrow\zeta}u(x)
\end{equation}
with the possibilities $\infty \leq \infty$ and  $-\infty \leq -\infty$ excluded, then $u \geq v$ almost everywhere in $\Omega$.
\end{theorem}

For more basic properties of $n$-superharmonic functions, we refer readers  to \cite{HeKi, KM-88, HKM, KM, Lind}.

%%%%%%%%%%%%%%%%%%%%%%%

\subsection{The story in 2 dimensions}\label{Subsect:2 dimensions revisit}

Thanks to the seminal paper \cite{Hu} of Huber in 1957 (see also \cite{Cv, Ff, BF, Hf}),  to explore the connection between
geometric properties of surfaces and potential theory based on Gauss curvature equations has been the major part of the theory 
of surfaces. The Gauss curvature equation in an isothermal coordinates on a surface is
\begin{equation}\label{Equ:gauss}
-\Delta u = Ke^{2u},
\end{equation}
where $K$ is the Gauss curvature of the surface metric $e^{2u}|dx|^2$. Let us focus on one thread of developments on this subject: 
local behavior of superharmonic functions near an isolated singular point or equivalently asymptotic behavior at infinity of 
superharmonic functions on the entire plane. 
\\

A function that is subharmonic on the entire plane is representable as a function of potential type
$$
v (z) =  \int_{\mathbb{C}} \log |1 - \frac z\xi|d\mu(\xi)
$$
for $z, \xi \in \mathbb{C}$ the complex plane, where $\mu$ is a positive mass distribution and vanishes in a neighborhood of the origin for our
purposes. To describe the asymptotic behavior of the function $v$ at infinity one aims to understand the limit
$$
\lim_{z\to \infty} \frac {v(z)}{\log |z|}.
$$

\vskip 0.1in
In this regard, notions of thinness play the natural and important role. Notions of thinness at 
a point was considered by Brelot in \cite{Br} in 1940, where a subset $E$ in $\mathbb{C}$ is said to be thin at a point $z_0$ if
either $z_0\notin \bar E$ or there exists a subharmonic function $v$ in a neighborhood of $z_0$ such that
\begin{equation}\label{Equ:Br-thin}
\limsup_{z\in E \text{ and } z\to z_0} v(z) < v(z_0),
\end{equation} 
which we will refer it as thinness by Cartan property (cf. \cite{AH}). 
This notion of thinness at a point is for potential functions with no point charge at the point.
\\

In \cite[(1.8)]{AH}, a subset $E$ of $\mathbb{C}$ is said to be thin at infinity if either it is bounded or there exists a function 
that is subharmonic on the entire complex plane $\mathbb{C}$  such that
\begin{equation}\label{Equ:AH-thin}
\limsup_{z \in E \text{ and } z \to \infty} \frac {v(z)}{\log |z|} < \limsup_{z\to\infty} \frac  {v(z)}{\log |z|}.
\end{equation}
At the end of \cite{AH}, there was a discussion about the correlation of these two notions 
of thinness. For a function $v$ of potential type, one may take an inversion and consider the subharmonic function 
$$
u (z) = v(\frac 1z) + M\log \frac 1{|z|}
$$
on the punctured plane with no charge at the origin, where $M$ is the total mass of the potential function $v$. 
Then \eqref{Equ:AH-thin} is equivalent to
\begin{equation}\label{Equ:Br-1-thin}
\limsup_{z\in \tilde E \text{ and } z\to 0}\frac {u(z)}{\log \frac 1{|z|}} < \limsup_{z\to 0}\frac {u(z)}{\log \frac 1{|z|}} = 0,
\end{equation}
where $\tilde E = \{\frac 1z: z\in E\}$. This is to say that the thinness defined in \cite[(1.8)]{AH} is the one for potential functions 
with point charge. It was then pointed out in \cite{AH} that $E$ is thin at infinity by \eqref{Equ:AH-thin} if and only if $\tilde E$ is thin at 
the origin by \eqref{Equ:Br-thin} thanks to \cite[Theorem 2]{Br-1}. These two types of thinness can be shown to be no longer equivalent 
for a nonlinear potential theory, which will be included in our future work. 
\\

In the geometric viewpoint, more interestingly, an equivalent criterion for a set to be thin 
at infinity using log-capacity was established as a Wiener type criterion in \cite{AH}.

\begin{theorem}\label{Thm:AH-thin} (\cite[Theorem 1.3]{AH}) Let $E$ be a Borel subset set in the plane and $\gamma_n$ be the logarithmic 
capacity of the part of $E$ lying in the annulus $\{z\in \mathbb{C}: r^n < |z| \leq r^{n+1}\}$ for a fixed number $r>1$ and 
$n=1, 2, 3, \cdots$. Then $E$ is thin at infinity if and only if $\gamma_n\to 0$ as $n\to\infty$ and 
\begin{equation}\label{Equ:Wiener}
\sum_{n=1}^\infty \frac n{\log \frac 1{\gamma_n}} < \infty.
\end{equation}
\end{theorem}

In summary, based on works in \cite{Br, Br-1, AH} (see also \cite{Hen, Hu-1, Ar, Hay, HK}), one knows that, for a function $v$ of 
potential type, there is a set $E$ that is thin at infinity and
\begin{equation}\label{Equ:asymptotic}
\lim_{z \notin E \text{ and } z \to \infty} \frac {v(z)}{\log |z|} = \limsup_{z\to\infty} \frac  {v(z)}{\log |z|}.
\end{equation}
Naturally, one asks what is the condition for a function of potential type to have a clean asymptotic behavior \eqref{Equ:asymptotic} 
with no exception thin set $E$?
It is until very recent this question was solved analytically in \cite[Theorem 2.1]{Tal-2} in 2006 and geometrically in 
\cite[Lemma 4.2]{BMQ-s} in 2016.  Namely,

\begin{theorem*} (\cite[Lemma 4.2]{BMQ-s}) Suppose that $(\mathbb{C}, e^{2u}|dz|^2)$ is complete with nonnegative and bounded Gauss 
curvature. Then
\begin{equation}\label{Equ:asymptotic-2d}
u (z) = m\log \frac 1{|z|} + o(\log |z|) \text{ as $|z|\to\infty$}
\end{equation}
for $m\in [0, 1]$.
\end{theorem*}

\noindent
It is known that $m = \frac 1{2\pi}\int_{\mathbb{C}} Ke^{2u}dz$ and $m\in [0, 1]$ due to \cite{Cv, Hu}, where $m=0$ implies
$u$ is a constant. The proof of the above result in \cite{BMQ-s} relies on two important ingredients that are deep in geometric analysis 
and partial differential equation. One is the non-collapsing result of Croke-Karcher \cite[Theorem A]{CK} in 1988 for complete surfaces 
with nonnegative Gauss curvature; the other is asymptotic estimates for nonnegative solutions to Gauss curvature type equations of 
Taliaferro in \cite[Theorem 2.1]{Tal-2} (see also his previous work \cite{Tal, Tal-1}). One of the key analytic ingredients in 
\cite{Tal, Tal-1, Tal-2} is the Brezis-Merle inequality of Moser-Trudinger type \begin{equation}\label{Equ:Brezis-Merle} 
\int_{\Omega} e^{\frac {(4\pi - \delta)|u (x)|}{\|\Delta u\|_{L^1(\Omega)}} }dx \leq (\text{diam}(\Omega))^2\frac {4\pi^2}{\delta}
\end{equation}
for $u |_{\partial\Omega} = 0$ and $\delta\in (0, 4\pi)$, established in \cite[Theorem 1]{BM} (cf. \cite{Io}).
\\

Taliaferro's estimates in \cite{Tal, Tal-1, Tal-2} are the major work in the theory of local behavior 
of a class of subharmonic functions near an isolated singular point. And, in the spirit of Huber that was reflected in \cite{Hu}, 
on geometric side, it was a very successful story that the above theorem of sharp 
local behavior (cf. \cite[Lemma 4.2]{BMQ-s}) turns out to be essential to the proof of \cite[Main Theorem]{BMQ-s} in 2 dimensions 
that a complete, non-negatively curved, immersed surface in hyperbolic 3-space is necessarily properly embedded, except 
coverings of equidistant surfaces, which was conjectured by Epstein and Alexander-Currier in \cite{AlCu, AlCu2, Eps, Eps1, Eps2} 
around 1990. 

%%%%%%%%%%%%%%%%%%%%%%

\subsection{Isolated singularity for nonnegative $n$-superharmonic functions}\label{Subs:verons}

There have been significant developments of the study on local and global behaviors for solutions to (degenerate) quasilinear elliptic 
equations that include the study of $n$-Laplace equations, for example, \cite{KV, B-V, Ve} and references therein. The following result
on the isolated singularities of nonnegative $n$-superharmonic functions are usually of importance.

\begin{theorem} \label{Thm:B-V}
(\cite[Proposition 1.1]{B-V}) Let $0<r<R$. Suppose that $w$ is a nonnegative $n$-superharmonic function on the punctured ball
$B(0, R)\setminus\{0\}$. Assume that $w$ is continuous and $|\nabla w|^n, -\Delta_n w$ is locally integrable in $B(0, R)\setminus\{0\}$.  
 Then, if  $\lim_{x\to 0} w(x) = \infty$, then there are a function $g\in L^1(B(0, r))$
and a number $\beta \geq 0$ such that in 
\begin{equation}\label{Equ:B-Veron}
-\Delta_n w = g + \beta \delta_0, x\in B(0,r)
\end{equation}
in the distributional sense, where $\delta_0$ is the Dirac function at the origin. And $|\nabla u|\in L^p(B(0,r))$ for $p\in (0,n)$.
\end{theorem}

\begin{remark}\label{Rem:vis-radon}
We remark here that the function $w$ in the above theorem is in fact an $n$-superharmonic in the ball $B(0, 1)$ as the potential of 
the nonnegative Radon measure that is induced from $g + \beta\delta$.
\end{remark} 

%However, when we deal with n-superharmonic function in the sense of Definition \ref{Def:superh}, the following result is useful to us.
%\begin{theorem} \label{radon measure} (Theorem 2.1 of \cite{KM-92})
%Suppose $u$ is n-superharmonic in $\Omega$. Then there is nonnegative Radon measure $\mu$ in $\Omega$ such that $-\Delta_n u=\mu$.
%\end{theorem}

The other important contribution in the study of isolated singularity of $n$-harmonic functions is the following result in 
\cite{KV, KV-e}, which is based on previous works in \cite{Serrin-1, Serrin-2, Tol}.

\begin{theorem} \label{Thm:KV}
(\cite[Theorem 1.1]{KV}) Suppose that $u$ is a nonnegative n-harmonic function on the punctured ball $B(0, r_0)
\setminus\{0\}$. Then, there is a number $\gamma$ such that 
\begin{equation}\label{Equ:k-v}
u (x) - \gamma\log \frac 1{|x|} \in L^\infty_{\text{loc}}(B(0, r_0)).
\end{equation}
\end{theorem}

The idea of the proof of this theorem in \cite{KV} is particularly helpful to us. In fact, in some sense, what we would like to have
is the extension of this theorem to cover $n$-superharmonic functions. Our approach combines that in \cite{KV} and the use of the 
nonlinear potential theory \cite{HKM, KM, PV, AM, Lind}. 

%%%%%%%%%%%%%%%%

\subsection{Non-linear potential theory for $n$-Laplace equations}\label{Subs:potential}

The nonlinear potential theory itself is a vast and profound subject in Mathematics. We certainly do not intend to give an comprehensive
introduction here. Instead we will collect useful facts in a cohesive way that we perceive. To study $n$-Laplace equations
\begin{equation}\label{Equ:n-Laplace}
- \Delta_n w = \mu,
\end{equation}
where $\mu$ is a nonnegative Radon measure representing the mass distribution, 
there is the nonlinear potential theory developed to replace the principle of superposition 
(cf. \cite{KM, HKM, HeKi, PV, AM, Lind}). The fundamental tool is the Wolff potential 
\begin{equation}\label{Equ:Wolff}
W_{1, n}^\mu (x_0, r)  = \int_0^r \mu (B(x_0,  t))^{\frac 1{n-1}} \frac {dt}t.
\end{equation}
The Wolff potential plays the same role in the nonlinear 
potential theory as the Riesz potential plays in the linear one.  And the foundational estimates in the nonlinear potential for the 
equation \eqref{Equ:n-Laplace} is as follows:

\begin{theorem}\label{Thm:KM} 
(\cite[Theorem 1.6]{KM}) Suppose that $w$ is a nonnegative $n$-superharmonic function satisfying \eqref{Equ:n-Laplace}
for a nonnegative Radon measure $\mu$ in $B(x_0, 3r)$. Then
\begin{equation}\label{Equ:Key-potential}
C_1W^{\mu}_{1, n}(x_0, r)\leq w (x_0) \leq C_2 \inf_{B(x_0, r)} w + C_3W^{\mu}_{1, n}(x_0, 2r)
\end{equation}
for some dimensional constants $C_1, C_2, C_3 >0$.
\end{theorem}

It is easily seen that the study of $n$-Laplace equations is intimately related to $n$-capacity since solutions to $n$-Laplace 
equations are critical points for the functional $\int |\nabla u|^ndx$. We therefore recall the definition of $n$-capacity 
from \cite[Section 3]{KM}. 

\begin{definition}\label{Def:n-capacity}
For a compact subset $K$ of a domain $\Omega$ in Euclidean space $\mathbb{R}^n$, we define
$$
\text{cap}_n(K, \Omega) = \inf \int_\Omega |\nabla u|^n dx
$$
for all $u$ runs through all $u\in C^\infty_0(\Omega)$ and $u\geq 1$ on $K$. Then $n$-capacity for arbitrary subset $E$
of $\, \Omega$ is
$$
\text{cap}_n (E, \Omega) = \inf_{\text{open $G\subset \Omega$ that contains $E$}}\sup_{\text{compact $K\subset G$}}
\text{cap}_n (K, \Omega).
$$
\end{definition}

$n$-capacity is clearly invariant under conformal transformations, and therefore is also called the conformal capacity 
(cf. \cite{G}, for example). The notions of $n$-thinness in the potential theory are important in the study of $n$-superharmonic functions. 
Notions of $n$-thinness were first considered in \cite{AM}, and readers are referred to \cite{AM, HeKi, KM} for more background 
and references. One notion of $n$-thinness is defined via Wiener integral given in \cite{AM, KM}, which we will 
refer to as thinness by Wiener integral. One of the major achievements in \cite{KM} is to establish the complete equivalence 
between the thinness by Wiener integral and the one by Cartan property \eqref{Equ:Br-thin} in general dimensions, based on 
\cite[Theorem 1.6]{KM} and early works \cite{AM, HeKi}. But, these notions of thinness at a point are for potential functions with 
no point charge, which is only known to be the same as the notion of thinness for potential functions with point charge in 2 
dimensions (\cite[Theorem 2]{Br-1} and \cite{AH}) . In higher dimensions, inspired by \cite[Theorem 1.3]{AH}, 
we will introduce a new notion of thinness using $n$-capacity and study its relation to the Cartan property 
\eqref{Equ:AH-thin} for $n$-subharmonic functions at isolated singular point (cf. Definition \ref{Def:n-thin} and 
Theorem \ref{Thm:AH-1} in next section). 

%%%%%%%%%%%%%%%%

\subsection{$n$-Laplace equations in differential geometry}

What can we do in higher dimensions following the approach in \cite{Hu} by Huber? We have seen successful efforts in 
\cite{SY, Zhu, CQY, CEOY, CH, CH-1} to explored higher dimensional counterparts of Gauss curvature equations \eqref{Equ:gauss}
such as the scalar curvature equations 
$$
 - \frac {4(n-1)}{n-2}\Delta u + R u = \bar R u^\frac {n+2}{n-2},
$$
where $R$ and $\bar R$ are scalar curvature of the metrics $g$ and $\bar g = u^\frac 4{n-2} g$ respectively in dimensions $n\geq 3$;
and the higher order analogue: $Q$-curvature equations, 
$$
P_n w + Q_n = \bar Q_n e^{2nw},
$$
where $P_n = (-\Delta )^n + \ lower \ order$ is the so-called Paneitz type operator and $Q_n, \bar Q_n$ are so-called $Q$-curvature
of the metrics $g$ and $\bar g = e^{2w}g$ respectively in dimensions $2n\geq 2$.  
We have also seen remarkable successes in using fully nonlinear equations of Weyl-Schouten 
curvature, as replacements of Gauss curvature equations, in \cite{CGY, CHY, GLW, mG}. The above mentioned seem to represent 
major developments in conformal geometry and conformally invariant partial differential equations following the approach in \cite{Hu} 
by Huber. 

%%%%%%%%

\subsubsection{$n$-Laplace equations in conformal geometry}
Recall the change of Ricci curvature under conformal change of metrics is
$$
\bar R_{ij} = R_{ij} - \Delta \phi g_{ij} + (2-n) \phi_{i,j} + (n-2) \phi_i\phi_j + (2-n) |\nabla\phi|^2 g_{ij},
$$
where $R_{ij}, \bar R_{ij}$ are Ricci curvature tensors for the metrics $g$ and $\bar g = e^{2\phi}g$ respectively in $n$ dimensions. 
Contracting with $\phi^i$ and $\phi^j$ on both sides of the above equation, one gets that
$$
\phi^i\phi^j\bar R_{ij} = \phi^i\phi^jR_{ij} - |\nabla\phi|^{4-n}\text{div}(|\nabla\phi|^{n-2}\nabla\phi).
$$
Therefore one arrives at another generalization of Gauss curvature equations in higher dimensions, 
\begin{equation}\label{Euq:intro-n-Laplace-g}
- |\nabla\phi|^{2-n}\text{div}(|d\phi|^{n-2}\nabla\phi) + Ric(\frac{\nabla\phi}{|\nabla\phi|}, \frac{\nabla\phi}{|\nabla\phi|}) = \bar Ric (
\frac{\bar\nabla\phi}{|\bar\nabla\phi|}, \frac{\bar\nabla\phi}{|\bar\nabla\phi|})e^{2\phi}.
\end{equation}
Particularly, when $g$ is Ricci-flat, we have
\begin{equation}\label{Equ:intro-n-Laplace}
- \Delta_n \phi = \bar Ric (\frac{\bar\nabla\phi}{|\bar\nabla\phi|}, \frac{\bar\nabla\phi}{|\bar\nabla\phi|})e^{2\phi}|\nabla\phi|^{n-2}.
\end{equation}
In this paper we want to explore properties of $n$-superharmonic functions and the geometric consequences.  
Following the approach in \cite{Hu} by Huber we want to extend the success in 2 dimensions to higher dimensions and 
complement contemporary developments in conformal geometry and conformally invariant partial differential equations.

%%%%%%%

\subsubsection{Hypersurfaces in hyperbolic space}\label{Subsect:hypersurfaces}
Apparently, the first use of $n$-subharmonic functions in differential geometry was in \cite{BMQ-r} to overcome the limitation of the 
use of subharmonic functions in 2 dimensions or sectional curvature assumptions. 
Inspired by the calculation in \cite{AlCu, AlCu2}, it was calculated and concluded

\begin{theorem*} (\cite[Theorem 3.1]{BMQ-r}) The height function in Busemann coordinates for a hypersurface with nonnegative 
Ricci curvature in hyperbolic space is $n$-subharmonic.
\end{theorem*}

It is perhaps worth to mention, for immersed hypersurfaces $\Sigma^n \subset \h^{n+1}$ with appropriate orientation,  the following 
successively stronger pointwise convexity conditions on the principal curvatures $\kappa_1,\dots,\kappa_n$: 
$$
\begin{array}{ll}
\kappa_i >  0 & \text{\em strict convexity}\\
\kappa_i (\sum_{l=1}^{n}\kappa_l) -  \kappa_i^2 \geq n-1  & \text{\em nonnegative Ricci curvature}\\
\kappa_i \kappa_j \geq 1 & \text{\em nonnegative sectional curvature for $i\neq j$}
\end{array}
$$ 
This observation enables the authors in \cite{BMQ-r} to improve the end structure theorem of \cite{AlCu, AlCu2} as follows:

\begin{theorem*} (\cite[Main Theorem]{BMQ-r}) For $n \geq 3$, suppose that $\Sigma$ is a complete and noncompact 
hypersurface with nonnegative Ricci curvature properly embedded in hyperbolic space $\mathbb{H}^{n+1}$. Then 
$\partial_\infty\Sigma$ consists of at most two points. The case that  $\partial_\infty\Sigma$ consists of two points is 
a rigidity condition that forces $\Sigma$ to be an equidistant hypersurface about a geodesic line. 
\end{theorem*}

In this paper we will use the properties of $n$-superharmonic functions to derive asymptotic behaviors for hypersurfaces in 
herperbolic space with nonnegative Ricci and improve the asymptotic results in \cite{AlCu, AlCu2}. 

%%%%%%%%%%%%%%%%%%%%%%%%%%%%%%%%%%%%%%%%%%%%%%%%%%%%%%
%%%%%%%%%%%%%%%%%%%%%%%%

\section{Higher dimensional Arsove-Huber's theorem}\label{Sec:Arsove-Huber}

In this section our goal is to extend Theorem \ref{Thm:AH-thin} and \eqref{Equ:asymptotic} (cf. \cite[Theorem 1.3]{AH}) 
in general dimensions. First we define a notion of thinness by capacity inspired by that in 2 dimensions in 
\cite[Theorem 1.3]{AH} for potential functions with point charge.
\\

For $x_0\in\mathbb{R}^{n}$, we set
\begin{align*}
\omega(x_0, i) & =\{x\in \mathbb{R}^n: \ 2^{-i-1}\leq|x - x_0|\leq2^{-i}\} \text{ and }\\
\Omega(x_0, i) & =\{x\in \mathbb{R}^n: \ 2^{-i-2}<|x - x_0|<2^{-i+1}\}.
\end{align*}
And we set
\begin{align*}
\omega(\infty,i) & =\{x\in \mathbb{R}^n: \ 2^{i}\leq|x|\leq2^{i+1}\} \text{ and } \\
\Omega(\infty,i) & =\{x\in \mathbb{R}^n: \ 2^{i-1}<|x|<2^{i+2}\}.
\end{align*}

\begin{definition}\label{Def:n-thin}
Let $E\subset\mathbb{R}^{n}$ and $x_0\in\mathbb{R}^{n}$. We say $E$
is $n$-thin at $x_0$, if 
\[
\sum_{i=1}^{\infty}i^{n-1}cap_{n}(E\cap\omega(x_0,i),\Omega(x_0,i))<+\infty.
\]
Clearly $E$ is trivially $n$-thin if $x_0\notin \bar E$. Similarly, we say $E$ is $n$-thin at $\infty$, if 
\[
\sum_{i=1}^{\infty}i^{n-1}cap_{n}(E\cap\omega(\infty,i),\Omega(\infty,i))<+\infty.
\]
Again, $E$ is trivially $n$-thin at $\infty$ if $E$ is bounded. 
\end{definition}

Clearly the inversion $\frac x{|x|^2}$ of $\mathbb{R}^n$ takes a subset $E\subset\mathbb{R}^n$ that is $n$-thin at infinity 
to a subset $\tilde E$ that is $n$-thin at the origin.

\begin{theorem}\label{Thm:AH-1}
Let $w$ 
be a nonnegative lower semi-continuous function that is $n$-superharmonic in $B(0, 2)\subset \r^n$ and 
\[
-\Delta_n w=\mu\ge 0
\]
for a Radon measure $\mu\ge 0$.
Then there is a set $E\subset\mathbb{R}^n$, which is $n$-thin at the origin, such that
\[
\lim_{x\notin E \text{ and } |x|\rightarrow 0}\frac{w(x)}{\log\frac{1}{|x|}}=\liminf_{|x|\rightarrow0}\frac{w(x)}{\log\frac{1}{|x|}}=m\ge 0
\]
and
\[
w(x)\geq m\log\frac{1}{|x|} - C \ \text{ for $x\in B(0, 1)\setminus \{0\}$ and some $C$}.
\]
Moreover, if $w\in C^2(B(0,2)\setminus \{0\})$ and $(B(0, 2)\setminus \{0\},e^{2w}|dx|^2)$ is complete at the origin, then $m\geq1$. 
\end{theorem}
 To study the local behavior for a nonnegative
$n$-superharmonic function $w$ on the punctured ball, we follow the idea from \cite{KV} to 
consider the blow-down
\begin{equation}\label{Equ:blow-down}
w_r(\xi) = \frac {w(r\xi)}{\log \frac 1r} \text{ for $\xi \in B(0, \frac {2}r)\setminus \{0\}$ as $r\to0$}.
\end{equation}

%%%%%%%%%%%%%%%%%%%%

\subsection{The first step in the proof of Theorem \ref{Thm:AH-1}}\label{Subsect:step-1-AH}
The first we need is that the quotient $\frac {w(x)}{\log \frac 1{|x|}}$ is mostly uniformly bounded. 
Therefore the following proposition is the first key step to prove Theorem \ref{Thm:AH-1}.

\begin{proposition}\label{Prop:unbdd-thin} Assume the same assumptions as in Theorem \ref{Thm:AH-1}.
Then, there is a set $\hat E$, which is $n$-thin at the origin, and a constant $\hat C$ such that 
\begin{equation}\label{Equ:bound-blow-down}
0\leq \frac {w(x)}{\log \frac 1{|x|}}\leq \hat C
\end{equation}
for $x\in (B(0, 1)\setminus\{0\})\setminus \hat E$. 
\end{proposition}

The proof of Proposition \ref{Prop:unbdd-thin} starts with the following simple fact observed in \cite[Lemma 3.9]{KM}.

\begin{lemma}\label{Lem:lemma 3.9} (\cite[Lemma 3.9]{KM}) Suppose that $u$ is an $n$-superharmonic 
function satisfying $\min\{u,\lambda\}\in W_0^{1,n}(\Omega)$ for $\forall \lambda>0$ and
$$
-\Delta_n u = \mu
$$
for a nonnegative Radon measure $\mu$. Then, for $\lambda >0$,
\begin{equation}\label{Equ:n-capacity}
\lambda^{n-1}\text{cap}_n(\{x\in\Omega: u (x) > \lambda\}, \Omega) \leq \mu(\Omega).
\end{equation}
\end{lemma}

The proof of Lemma \ref{Lem:lemma 3.9} is to use $\frac{\min\{u, \lambda\}}\lambda$ as a test function and get
$$
\int_\Omega |\nabla \min\{u, \lambda\}|^n \leq \frac {\mu(\Omega)}{\lambda^{n-1}},
$$
which is easily seen to imply the above $n$-capacity estimate \eqref{Equ:n-capacity}. 
The next fact we need to prove Proposition \ref{Prop:unbdd-thin} is the following basic existence result (cf. \cite[Theorem 2.4]{KM-92}).

\begin{lemma}\label{Lem:basic existence} (\cite[Theorem 2.4]{KM-92}) 
For a bounded domain $\Omega\subset\mathbb{R}^n$ and a nonnegative finite Radon measure $\mu$, there always exists a solution $u(x) \ge 0$ to the equation
$$-\Delta_n u=\mu \text{ in $\Omega$}
$$
and $\min\{u,\lambda\}\in W^{1,n}_0(\Omega)$ for any $\lambda>0$.
\end{lemma}

To make use of the fundamental estimates \eqref{Equ:Key-potential} in Theorem \ref{Thm:KM}(cf. \cite[Theorem 1.6]{KM}), 
we also need the following estimates on the infimum.

\begin{lemma}\label{Lem:inf estimate}
Suppose that $w\ge0$ satisfies 
\[
-\Delta_{n} w  =\mu\ge0.
\]
Then there is a constant $C>0$ such that
\begin{equation} \label{Equ:inf estimate}
\inf_{B(x_{0}, \frac{|x_{0}|}{2})} w(x) \leq C(n,\|w\|_{L^n(B(0,1))}) \log\frac{1}{|x_{0}|}
\end{equation} 
for each $x_0\in B(0, \frac{1}2)\setminus\{0\}$.
\end{lemma}

\begin{proof} We will rely on some estimates from  \cite[Section 7]{DHM} to derive this lemma. 
Readers are referred to \cite{DHM} for definitions and notations. Particularly, in the light of \cite[Lemma 14 and 15]{DHM}, we know that 
\begin{equation}\label{Equ:bmo}
\|w\|_{BMO(B (0, \frac14), B (0, \frac12))} \leq C_1(n) (1+\|w\|_{L^{n}(B (0, \frac12))}).
\end{equation}
Meanwhile, from \cite[Theorem 5.11]{Lind}, for instance,  we know $\|w\|_{L^{n}(B (0, \frac12)}$ is finite. 
Therefore $\|w\|_{BMO(B (0, \frac14), B (0, \frac12))}$ is finite. Here we remark that in \cite[Section 7]{DHM}, the assumption that 
the right hand side in $L^1$ can be generalized to a nonnegative Radon measure easily.  And the assumption that $w\in W^{1,n}$ is not 
essential,  because, if not, we can replace $w$ by $\min\{u,k\}$, which belongs to $W^{1,n}$,  and use Fatou's lemma to 
prove \eqref{Equ:bmo} for $w$. 
\\

Suppose otherwise that  \eqref{Equ:inf estimate} were not true. Then we have a sequence 
$$
\{x_i\in B(0, \frac 14): x_i \rightarrow 0\}
$$
such that 
\begin{equation}\label{Equ:otherwise}
\inf_{B(x_{i}, \frac 12 |x_i|)}w (x) >C(n,\|w\|_{L^n(B(0,1))})\log\frac{1}{|x_{i}|}, 
\end{equation}
for some $C(n,\|w\|_{L^n(B(0,1))})$ to be fixed.
We let 
\[
\mu_{i}=\mu(\{x\in B(0, \frac 14): w(x)-\bar{w} > \frac{C(n,\|w\|_{L^n(B(0,1))})}{2}\log\frac{1}{|x_{i}|}\}),
\]
where the finite number $\bar{w}$ is the average of $w$ on $B(0, \frac14).$
Clearly, at least for $i$ large, 
\begin{equation}\label{Equ:contradiction}
\mu_{i}\geq C(n)|x_{i}|^{n}, 
\end{equation}
because of \eqref{Equ:otherwise}. On the other hand, From \cite[Lemma 1]{JN}, we know that, there are $B(n)$ and $b(n)$, such that 
\begin{align*}
\mu_{i} & \leq B(n)e^{-\frac{b(n)C(n,\|w\|_{L^n(B(0,1))})\log\frac{1}{|x_{i}|}}{2\|w\|_{BMO(B(0, \frac 14), B(0, \frac 12))}}}\mu(B(0, \frac 12))\\
 & \leq B(n)\mu(B(0, \frac 12))|x_{i}|^{\frac{b(n)C(n,\|w\|_{L^n(B(0,1))})}{2\|w\|_{BMO(B(0, \frac 14), B(0, \frac 12))}}}.
\end{align*}
Now we choose 
$$C(n,\|w\|_{L^n(B(0,1))})=2(n+1)b(n)^{-1}C_1(n)(1+\|w\|_{L^n(B(0,1))})
$$
and will get $\mu_i\leq B(n)\mu(B(0,\frac{1}2))|x_i|^{n+1}$, which is a contradiction with (\ref{Equ:contradiction}). Thus the proof is completed.
\end{proof}

For convenience and simplicity, we use
\begin{align*}
\omega_i & = \omega(0, i) = \{x\in \mathbb{R}^n: \ 2^{-i-1}\leq|x|\leq2^{-i}\} = 2^{-i}\omega(0, 0)\\
\Omega_i & = \Omega(0, i) =\{x\in \mathbb{R}^n: \ 2^{-i-2}<|x|<2^{-i+1}\} = 2^{-i}\Omega(0, 0).
\end{align*}
Then
$$
\frac i2 \log 2 \leq (i-2)\log 2  \leq \log\frac 1{|x|} \leq (i+1)\log 2  \leq 2i\log 2 \quad  \text{ for all $x\in \Omega_i$}.
$$
Now we are ready to start the proof of Proposition \ref{Prop:unbdd-thin}.
\\

\begin{proof}[The proof of Proposition \ref{Prop:unbdd-thin}] It is obvious that
$\frac{w(y)}{\log\frac{1}{|y|}}\geq0$. We are going to prove that
outside some set $\hat{E}$, which is $n$-thin at the origin, the quotient $\frac{w(y)}{\log\frac{1}{|y|}}$
has upper bound. 
\\

We cover $\omega_0$ with finite number of balls $\{B_{1}^{0},\cdots,B_{m}^{0}\}$, where the center of $B_j^0$ lies in $\omega_0$,  
the concentric ball $4B_{j}^{0}\subset \Omega_0$ for $j=1,\cdots,m$,  and $m$ depends only on the dimension $n$.
For $i\geq0$, we denote $B_{j}^{i}=2^{-i}B_{j}^{0}.$ It's obvious
that $\{B_{j}^{i}: j=1,\cdots,m\}$ cover $\omega_{i}$ and each $4B_{j}^{i}$
lie in $\Omega_{i}$. We let $r_{ij}$ be the radius of $B_{j}^{i}.$
Clearly $r_{ij} = 2^{-i} r_{0j}$. 
\\

For any $y\in B_{j}^{i}$, from  \cite[Theorem 1.6]{KM} and Lemma \ref{Lem:inf estimate}, we have
$$
w(y) \leq C_{2}(n) \inf_{B(y,\frac{|y|}{8})}w+C_{3}(n)W_{1,n}^{\mu}(y,\frac{|y|}{4})
$$
Since $|y|\sim r_{ij}\sim2^{-i}$ and 
\[
|W_{1,n}^{\mu}(y,\frac{|y|}{4})-W_{1,n}^{\mu}(y, \frac 13{r_{ij}})| = 
|\int_{\frac 13r_{ij}}^{\frac{|y|}{4}}\mu(B(y,t))^{\frac{1}{n-1}}\frac{dt}{t}|
\leq C.
\]
We arrive at
\begin{equation}\label{Equ:need wolff}
 w(y) \leq C(\log\frac{1}{|y|}+W_{1,n}^{\mu}(y, \frac 13r_{ij}) +1),
\end{equation}
To estimate $W_{1, n}^\mu(y, \frac 13r_{ij})$, we use Lemma \ref{Lem:basic existence} and solve the following 
\[
\begin{cases}
-\Delta_{n}w_{ij}(y) & =\mu,\,\,{\rm in}\,\,2B_{j}^{i}(y)\\
w_{ij}(y)|_{\partial(2B_{j}^{i})} & =0.
\end{cases}
\]
The advantage is that, from \cite[Lemma 3.9]{KM}, we know that 
\[
cap_{n}(\{y\in B_{j}^{i}: w_{ij}(y)>\log\frac{1}{|y|}\}, 2B_{j}^{i})\leq\frac{C\mu(2B_{j}^{i})}{i^{n-1}}.
\]
Now, using \cite[Theorem 1.6]{KM} again, we have 
\[
C_{1}W_{1,n}^{\mu}(y, \frac 13r_{ij}) \leq w_{ij}(y),\forall y\in B_{j}^{i},
\]
which implies that 
\begin{equation}\label{Equ:potential-estimate}
cap_{n}(\{y\in B_{j}^{i}: W_{1,n}^{\mu}(y,\frac 13r_{ij})>\frac{1}{C_{1}}\log\frac{1}{|y|}\},2B_{j}^{i})\leq\frac{C\mu(2B_{j}^{i})}{i^{n-1}}.
\end{equation}
Let
\[
\hat{E}_{ij}=\{y\in B_{j}^{i}: W_{1,n}^{\mu}(y,\frac 13 r_{ij})>\frac{1}{C_{1}}\log\frac{1}{|y|}\}\cap\omega_i
\]
 and 
\begin{equation}
\hat{E}_{i}=\cup_{j}\hat{E}_{ij} \quad \hat{E}=\cup_{i}\hat{E}_{i}.\label{E definition}
\end{equation}
Then we have
\[
cap_{n}(\hat{E}_{ij},\Omega_i)\le cap_n(\hat{E}_{ij},2B^i_j)\leq\frac{C\mu(2B_{j}^{i})}{i^{n-1}}.
\]
Hence from Theorem 2.2 (vi) of \cite{HKM}
\[
cap_{n}(\hat{E}\cap\omega_i,\Omega_i)\le \sum_j cap_n(\hat{E}_{ij},\Omega_i)\leq\frac{C\mu(\Omega_i)}{i^{n-1}}.
\]
Therefore
\[
\sum_{i}i^{n-1}cap_{n}(\hat{E}\cap\omega_i,\Omega_i)\leq C\mu(B_{1}(0)\backslash\{0\})<+\infty.
\]
Thus, from \eqref{Equ:need wolff}, there is a constant $\hat C > 0$ such that, outside $\hat{E}$, 
which is $n$-thin according to Definition \ref{Def:n-thin}, \eqref{Equ:bound-blow-down} holds.
The proof is completed.
\end{proof}

%%%%%%%%%%%%%%%%%

\subsection{The second step in the proof of Theorem \ref{Thm:AH-1}}\label{Subsect:step-2-AH}
The second key step in the proof of Theorem \ref{Thm:AH-1}, for the sake of the blow-down argument as the one 
used in \cite{KV}, is to modify the function $\frac {w(r\xi)}{\log\frac 1{r}}$ to accommodate the lack of boundedness. We use 
the trick from \cite{DHM} and consider the cut-off function
$$
a_\alpha (s) = \left\{\aligned s \quad\quad & \text{ when $0\leq s \leq \alpha$}\\ \alpha + \int_\alpha^s (\frac \alpha t)^\frac n{n-1}dt 
\quad & \text{ when $s > \alpha$}, \endaligned\right.
$$
where $\alpha$ is to be fixed as $\hat C+1$ throughout this paper, where $\hat C$ is the one in \eqref{Equ:bound-blow-down}. 
One may calculate that 
\begin{align}
a_\alpha (s) & \leq n\alpha \label{Equ:a_alpha calculation-1}\\
a_\alpha' (s)  & = \left\{\aligned 1 \quad\quad & \text{ when $0\leq s \leq \alpha$}\\ 
                                                  (\frac \alpha s)^\frac n{n-1} \quad & \text{ when $s > \alpha$}, 
                            \endaligned\right. \label{Equ:a_alpha calculation-2}\\
a_\alpha ''(s) & =  \left\{\aligned 0 \quad\quad & \text{ when $0\leq s \leq \alpha$}\\ 
                                                  - \frac n{n-1} (\frac \alpha s)^\frac n{n-1} s^{-1} \quad & \text{ when $s > \alpha$}, 
                             \endaligned\right. \label{Equ:a_alpha calculation-3}\\
- \Delta_n a_\alpha (u) & = \left\{\aligned - \Delta_n u \quad\quad & \text{ when $0\leq u \leq \alpha$}\\
                                                               - (\frac \alpha u)^n \Delta_n u + n (\frac \alpha u)^n u^{-1} |\nabla u|^n \quad & \text{ when $u > \alpha$}. 
                                                   \endaligned\right. \label{Equ:a_alpha calculation-4}
\end{align}

Now we are to carry out the blow-down argument as in \cite{KV}. 
For each $r>0$ and small, we consider the modified blow-down
\begin{equation}\label{Equ:modification}
\hat w_r (\xi) = a_\alpha (w_r(\xi)) = a_\alpha (\frac {w(r\xi)}{\log\frac 1r}).
\end{equation}
Clearly, we have
\begin{equation}\label{Equ:bounded angle}
0\leq \hat w_r(\xi) \leq n\alpha = n(\hat C+1) 
\end{equation}
for 
\begin{equation}\label{Equ:annulus}
\xi \in A_{0, \frac 1r}= \{\xi\in \mathbb{R}^n: |\xi|\in (0, \frac 1r)\}
\end{equation} 
and 
\begin{align}\label{Equ:scaled n-Laplace}
- \Delta_n^\xi \hat w_r (\xi) = \left\{\aligned - \frac {r^n}{(\log \frac 1r)^{n-1}} \Delta^x_n w ( r\xi) \ & \text{for $0\leq w_r(\xi) \leq \alpha$}\\
                                                                     \frac {r^n}{(\log \frac 1r)^{n-1} } (\frac\alpha {w_r(\xi)} )^n( - \Delta^x_n w ( r\xi)
+ n \frac 1{w(r\xi)}  |\nabla^x w|^n (r\xi)) \  & \text{for $w_r(\xi) > \alpha$}
                                                   \endaligned\right.
\end{align}
for $\xi \in A_{0, \frac 1r}$. To summarize, we state the following lemma to use the above calculations.

\begin{lemma}\label{Lem:modification} Assume the same assumptions as in Theorem \ref{Thm:AH-1}.
Then the modified blow-down $\hat w_r(\xi)$ is a nonnegative and bounded $n$-superharmonic function satisfying
$$
-\Delta_n^\xi \hat w_r  = \hat \mu_r   \geq 0 \text{ in $A_{0, \frac 1r}$}
$$
for a Radon measure $\hat \mu_r in A_{0, \frac 1r}$ and $\hat w_r(\xi) \leq  n\hat C+n$ for all $x\in A_{0, \frac 1r}$. 
More importantly, for any fixed $R>1$,
\begin{equation}\label{Equ:n-harmonicity}
\int_{A_{\frac 1R, R}} d\hat \mu_r(\xi) \leq  (\frac 1{\log\frac 1r})^{n-1} \int_{A_{\frac rR, rR}} 
d\mu  + n \alpha^{n-1} \int_{A_{\frac rR, rR}\cap \hat E} 
\frac {|\nabla w|^n}{w^n}dx, 
\end{equation}
where $\hat E$ is the subset given in Proposition \ref{Prop:unbdd-thin}, which is $n$-thin at the origin.
\end{lemma}

By Lemma \ref{Lem:modification} we want to show that,  at least for sequences $r_k\to 0$,  $\hat w_{r_k}(\xi)$ converges to a bounded $n$-harmonic 
function on the entire space $\mathbb{R}^n$ except possibly the origin, which can only be a constant due to \cite{Red} because the origin and the infinity
are  removable singularities by \cite{Serrin-1}. To be more precise, we need the following convergence lemma.

\begin{lemma}\label{Lem:convergence lemma}
Suppose that $\{u_i\}$ is a sequence of $n$-superharmonic functions in a bounded domain $\Omega\subset\mathbb{R}^{n}$ and
\[
-\Delta_{n}u_{i}=\mu_{i} \,\,\rm{in} \ \Omega,
\] 
where $\mu_{i}$ is a sequence of Radon measures.  
Assume that 
$$
0\leq u_{i}\leq M \ \text{ and } \mu_i \to 0 \ \text{ in the sense of distribution }.
$$
Then, for each bounded subdomain $D\subset \bar D\subset \Omega$, there is a constant $C>0$ such that
\begin{equation}\label{Equ:energy bound}
\int_D |\nabla u_i|^n dx \leq C
\end{equation}
for all $i$ and there is $u\in W^{1, n}(D)$ such that
$$
u_{i}\rightharpoonup u \ \text{ in } \ W^{1,n}(D) \ \text{ and } \ -\Delta_{n}u=0 \text{ in $D$ in distributional sense},
$$
taking a subsequence if necessary.
\end{lemma}

\begin{proof} For the convenience of readers, we present proof here. First we prove
\begin{equation}\label{Equ:uk estimate}
\int_D \frac {|\nabla u_i|^n}{(u_i+1)^2} dx \leq C \int_{\Omega\setminus D} (u_i+1)^{n-2}dx.
\end{equation}
Similar to the argument in \cite[Theorem 5.15]{Lind}, based on Lemma \ref{Lem:equiv def} (cf. \cite{HeKi} and \cite[Proposition 2.7]{HK}),
we simply use the testing functions $\zeta^n (u_i+1)^{-1}$, where
$$
\zeta\in C_{0}^{\infty}(\Omega),\zeta\equiv1\,\,{\rm on}\,\,D, |\nabla \zeta|\leq \frac {C}{dist(D,\partial \Omega)}.
$$  
Then, from 
\[
\int_{\Omega}-(\Delta_{n}u_{i})\zeta^{n}(u_{i}+1)^{-1}\geq0
\]
we get 
\[
\int_{\Omega}|\nabla u_{i}|^{n}\zeta^{n}(u_{i}+1)^{-2}\leq n^{n}\int_{\Omega}(u_{i}+1)^{n-2}|\nabla\zeta|^{n}.
\]
This obviously implies \eqref{Equ:uk estimate}. Next, to prove \eqref{Equ:energy bound} by \eqref{Equ:uk estimate},  
we derive
\begin{align*}
\int_{D}|\nabla u_{i}|^{n} & \leq(\sup|u_{i}|+1)^{2}\int_{\Omega}|\nabla (u_{i}+1)|^{n}\zeta^{n}(u_{i}+1)^{-2}\\
 & \leq n^{n}(\sup|u_{i}|+1)^{2}(\sup|u_{i}|+1)^{n-2}\int_{\Omega}|\nabla\zeta|^{n}\\
 & \leq C(n,\Omega,D,M).
\end{align*}
Hence there is $u\in W^{1,n}(D)$ such that $u_{i}\rightharpoonup u$ in $W^{1,n}(D)$, at least for a subsequence. In the light of 
\[
\int_{D}|\nabla u_{i}|^{n-2}<\nabla u_{i},\nabla\phi>=\int_{D}\phi d\mu_{i} \rightarrow 0
\] 
as $i\rightarrow\infty$ for any $\phi\in C_{0}^{\infty}(D)$, it suffices to prove that
\begin{equation}
\int_{D}|\nabla u_{i}|^{n-2}<\nabla u_{i},\nabla\phi>\rightarrow\int_{D}|\nabla u|^{n-2}<\nabla u,\nabla\phi> \label{uk to u}
\end{equation}
as $i\rightarrow \infty$. Thanks to \cite[Theorem 1.1]{Zheng}, we know that $u_{i}\rightarrow u$ strongly in $W^{1,p}(D)$ for
all $1\leq p <n$, which implies (\ref{uk to u}). Note that \cite[Theorem 1.1]{Zheng} imposed the condition that $-\Delta_n u\in L^1(\Omega)$. However, if one checks his argument carefully, the only place where this is used is when dealing with (2.7), Page 385. If we replace $f_k$ with nonnegative Radon measure $\mu_k$ with $\mu_k(\Omega)\to 0$, we can also prove that 
$$
|\int_{\Omega}w_k^{\lambda}d\mu_k|\le \lambda |\mu_k(\Omega)|\to 0.
$$
Thus the lemma is proved. 
\end{proof}

\begin{remark}\label{Rem:theorem 5.15} (\cite[Theorem 3.57]{HKM} \cite[Theorem 5.15]{Lind}) 
Let $u > 1$ be an $n$-superharmonic function in $\Omega$, which is not necessarily bounded from above. 
From the proof of \cite[Theorem 5.15]{Lind} (please see above), one actually 
gets 
\begin{equation}\label{Equ:5.15}
\int \zeta^n |\nabla u|^n u^{- 1 - \alpha} dx \leq C(n, \alpha) \int u^{n-1-\alpha} |\nabla\zeta|^ndx
\end{equation}
for any $\alpha \in (0, n-1]$ and any cut-off function as in the above proof. 
The right hand side of \eqref{Equ:5.15} is finite by \cite[Theorem 5.11]{Lind}. This remark is useful to handle the second term on the 
right side of \eqref{Equ:n-harmonicity}.
\end{remark}

%%%%%%%%%%%%%%%%%%%%

\subsection{The third step in the proof of Theorem \ref{Thm:AH-1}}\label{Subsect:step-3-AH}
The third key step in the proof of Theorem \ref{Thm:AH-1} is to show the uniqueness of possible limits of all blow-down sequences.
We continue to use the approach used as in \cite{KV}. One of the key tool is the following weak comparison principle as a consequence of \cite[Lemma 3.1]{Tol} (please also see \cite[Corollary 1.1]{KV} and the comment in \cite{KV-e}). 

\begin{lemma}\label{Lem:comparison} (\cite[Lemma 3.1]{Tol} \cite[Corollary 1.1]{KV}) Assume $\Omega$ is a connected open subset of $\mathbb{R}^n\setminus \{0\}$ 
and $u$ is $n$-superharmonic in $\Omega$. Then
\begin{equation}\label{Equ:weak comparison}
\inf_{\partial\Omega} \frac {u(x)}{\log\frac 1{|x|}} \leq \inf_{\Omega} \frac {u(x)}{\log\frac 1{|x|}}.
\end{equation} 
\end{lemma}
For any blow-down sequence $\hat w_{r_i}(\xi)$ with $r_i\to 0$, there is $\xi_{r_i}$ 
with $|\xi_{r_i}| =1$ and
\begin{equation}\label{Equ:minimum-r_k}
\hat w_{r_i} (\xi_{r_i}) = w_{r_i} (\xi_{r_i}) = \frac {w (r_i\xi_{r_i})}{\log\frac 1{r_i}} = \min_{|x| = r_i} \frac {w (x)}{\log\frac 1{|x|}} 
\to \liminf_{x\to 0}
\frac {w(x)}{\log\frac 1{|x|}} \le \hat C.
\end{equation}
Because, Lemma \ref{Lem:comparison} implies that the quotient $\min_{|x| = r}\frac {w(x)}{\log\frac 1{|x|}}$ is non-increasing as 
$r\to 0$, since the infimum is always achieved at the inner sphere of the annulus $B(0, r_0)\setminus B(0, s)$ for $r_0 < 1$ fixed while
$s$ arbitrarily small. Notice that we may assume 
$$
\lim_{|x| \to 1^-} \frac {w(x)}{\log\frac 1 {|x|}} = \infty
$$
if necessary. Because, when proving Theorem \ref{Thm:AH-1} one may deal with $w+\epsilon$ for arbitrarily small $\epsilon$ instead. 
We will present the proof of the uniqueness of all blow-down limits 
based on Lemma \ref{Lem:comparison} in the proof of Theorem \ref{Thm:AH-1} in next section.

%%%%%%%%%%%%%%%%%%%%%

\subsection{The last step of the proof of Theorem \ref{Thm:AH-1}}\label{Subsect:proof of Theorem 3.1}
With all the preparation we finally are ready to prove Theorem \ref{Thm:AH-1}. At this point, we have cleared almost everything 
except that the convergences of each blow-down sequence $\hat w_{r_k}$ to a constant is weaker than the
pointwise one. This in principle is caused by the fact that the density function is just a Radon measure $\mu$. 
Our main goal here, after presenting a proof of the uniqueness of 
the sequential blow-down limits, is to extract a possible bad set $E$, which is again $n$-thin so that outside 
$E$ the limit of the quotient $\frac{ w(x)}{\log\frac 1{|x|}}$ is $\liminf_{x\to 0} \frac{ w(x)}{\log\frac 1{|x|}}$ pointwisely. 

\begin{proof}[The proof of Theorem \ref{Thm:AH-1}] To recap, first, from Proposition \ref{Prop:unbdd-thin} in Section \ref{Subsect:step-1-AH}, 
we know that, outside the thin set $\hat{E}$, 
$$
\frac{w(x)}{\log\frac{1}{|x|}}\leq\hat{C}.
$$
Then, based on the discussion in Section \ref{Subsect:step-2-AH}, we consider the modified blow-down functions  
$\hat{w}_{r}(\xi)$ by \eqref{Equ:modification} for $\alpha=1+\hat{C}$. 
From Lemma \ref{Lem:convergence lemma} and Lemma \ref{Lem:modification},
for a sequence $r_i\to 0$, we may assume that $\hat{w}_{r_i}(\xi)$, 
converges to a bounded $n$-harmonic function $\hat w (\xi)$ in $A(0, \infty) = \mathbb{R}^n\setminus\{0\}$ 
(for some subsequence if necessary). When appying Lemma \ref{Lem:convergence lemma} and verifying $\mu_i\to 0$ in any compact subset of $\mathbb{R}^n\backslash\{0\}$, one needs 
to use \eqref{Equ:n-harmonicity} and Remark \ref{Rem:theorem 5.15}.
Thanks to Liouville type theorem of Reshetnyak \cite{Red}, $0$ and $\infty$ 
are removable singularities of $\hat{w} (\xi)$ and $\hat w(\xi) = \hat w$ is a constant. Finally, one would like to use
Lemma \ref{Lem:comparison} in Section \ref{Subsect:step-3-AH} to derive 
\begin{equation}\label{Equ:weak-limit}
\hat w = \gamma^-= \liminf_{r\to 0} \frac {w(x)}{\log\frac 1{|x|}}
\end{equation}
for any sequence $r_i\to0$. The remaining issue is that all the sequential convergences are only the one weak in $W^{1,n}$ and 
strong in $W^{1, p}$ for any $1\leq p < n$, which does not yet imply point-wise convergence as desired.  
\\

Now let us start with a proof of the uniqueness of $\hat w$ (i.e. \eqref{Equ:weak-limit}). 
Recall from \eqref{Equ:minimum-r_k}
$$
\hat w_{r_i} (\xi_{r_i}) \to \gamma^- = \liminf_{r\to 0} \frac {w(x)}{\log\frac 1{|x|}}
$$
for any sequence $r_i\to 0$.
Since $\hat{w}_{r_{i}}(\xi)$ converges to $\hat{w}$ strongly in $W^{1,p}(A(r_{0},\frac{1}{r_{0}})),1\leq p<n$ for any fixed small $r_0 > 0$, 
we know that
\[
\int_{B_{\frac 12}(\xi_{r_{i}})}(\hat{w}_{r_{i}}(\xi)-\gamma^{-})^{q}\rightarrow |B_{\frac 12}
(\xi_{r_{i}})|(\hat{w}-\gamma^{-})^{q}  \quad \text{ as $r_{i}\rightarrow0$}
\]
for any $0 < q < \infty$. By the way, $w_r (\xi) \geq \gamma^-$ due to the definition of $\gamma^-$ and Lemma \ref{Lem:comparison}. 
By invoking the weak Harnack inequality \cite[Theorem 3.51]{HKM}), we know
\[
\hat{w}_{r_{i}}(\xi_{r_i}) - \gamma^{-}\geq C(\int_{B_{\frac{1}{2}}(\xi_{r_{i}})}(\hat{w}_{r_{i}}(\xi)-\gamma^{-})^{q})^{\frac{1}{q}}
\]
for any $\xi \in B_{\frac 14}(\xi_{r_i})$ and some $0 < q < \infty$. 
Clearly this would be a contradiction if $\hat w \neq \gamma^-$. So this finishes the proof of the uniqueness for sequential blow-down 
limits.
\\

In the following, what we need to do is to refine the argument in the proof of Proposition \ref{Prop:unbdd-thin} to show that, 
outside an $n$-thin set, the quotient $\frac {w(x)}{\log\frac 1{|x|}}$ is not just bounded but actually convergent at the origin pointwisely. 
We will use the same notations and follow the same process. But we are in a better position than that we were in the proof of Proposition 
\ref{Prop:unbdd-thin}. First, we have the following improved \eqref{Equ:inf estimate} in Lemma \ref{Lem:inf estimate} 
\begin{equation}\label{Equ:improved inf estimate}
\lim_{ y \to 0} \inf_{x \in B(y, \frac 1{\alpha} |y|)} \frac {w(x)}{\log\frac 1{|x|}} = \gamma^-.
\end{equation}
This is because, from the uniqueness of all blow-down limits, we know
$$
\lim_{r\to 0} \hat w_r(\xi) = \gamma^-
$$
almost everywhere in $A_{r_0, \frac 1{r_0}}$ and that $\hat w_r$ and $w_r$ only differ at the set $\tilde E$ that is $n$-thin 
at the origin. In fact we have the following, which is even more useful.

\begin{lemma}\label{Lem:better inf estimate} Under the assumptions in Theorem \ref{Thm:AH-1}. 
\begin{equation}\label{Equ:better inf estimate}
\lim_{y \to 0} \frac {\inf_{B(y, \alpha |y|)} w(x)}{\log\frac 1{|y|}} = \gamma^-
\end{equation}
for any fixed $\alpha\in (0, 1)$.
\end{lemma}
\begin{proof}
First, if let
$$
\inf_{B(y, \frac 1\alpha |y|)} \frac {w(x)}{\log\frac 1{|x|}} = \frac {w(x_0)}{\log\frac 1{|x_0|}}
$$
for some $x_0\in \bar B(y, \alpha |y|)$, then $\frac {|y|}{|x_0|} \in [\frac 1{1+\alpha}, \frac 1{1-\alpha}]$ and
$$
\frac {\inf_{B(y, \alpha |y|)} w(x)}{\log\frac 1{|y|}} \leq \frac {w(x_0)}{\log\frac 1{|y|}} =
\frac {w(x_0)}{\log\frac 1{|x_0|}}\cdot \frac {\log\frac 1{|x_0|}} {\log\frac 1{|y|}} \leq \frac {w(x_0)}{\log\frac 1{|x_0|}}
( 1 + \frac {\log\frac 1{1-\alpha}}{\log\frac 1{|y|}}) .
$$
Next, if let 
$$
\inf_{B(y, \alpha |y|)} w(x) = w(y_0)
$$ 
for some $y_0\in \bar B(y, \alpha |y|)$, then $\frac {|y|}{|y_0|}\in [\frac 1{1+\alpha}, \frac 1{1-\alpha}]$ and
$$
\frac {\inf_{B(y, \alpha |y|)} w(x)}{\log\frac 1{|y|}} = \frac {w(y_0)}{\log\frac 1{|y|}} = \frac {w(y_0)}{\log\frac 1{|y_0|}} \cdot 
\frac {\log\frac 1{|y_0|}} {\log\frac 1{|y|}} \geq \gamma^- ( 1 + \frac {\log \frac 1{1+\alpha}}{\log\frac 1{|y|}}).
$$
Therefore, squeezing from both sides, we derive \eqref{Equ:better inf estimate}. The proof is completed. 
\end{proof}

Secondly, we apply \cite[Theorem 1.6]{KM} to $w(y) - \inf_{B(y, \frac 34|y|)} w$ in $B(y, \frac 34|y|)$ and obtain
\begin{equation}\label{Equ:2-KM94}
w (y) - \inf_{B(y, \frac 3 4 |y|)} w(x) \leq C_2\inf_{B(y, \frac 1 4 |y|)} (w - \inf_{B(y, \frac 34 |y|)} w) 
+ C_3 W_{1, n}^\mu (y, \frac12 |y|).
\end{equation}
Hence, 
$$
\frac {w(y)}{\log\frac 1{|y|}} \leq \frac {\inf_{B(y, \frac 34 |y|)} w(x)}{\log\frac 1{|y|}} 
+ C_2\frac {\inf_{B(y, \frac 14 |y|)} w}{\log\frac 1{|y|}} - C_2\frac{\inf_{B(y, \frac 34 |y|)} w}{\log\frac 1{|y|}} 
+ C_3 \frac {W_{1, n}^\mu(y, \frac 12|y|)}{\log\frac 1{|y|}}
$$
which implies, by \eqref{Equ:better inf estimate} in Lemma \ref{Lem:better inf estimate},
\begin{equation}\label{Equ:leave-potential}
\limsup_{y\to 0} \frac {w(y)}{\log\frac 1{|y|}} \leq \gamma^- + C_3 \limsup_{y\to 0} \frac {W_{1, n}^\mu(y, \frac 12 |y|)}{\log\frac 1{|y|}}.
\end{equation}

Thirdly, regarding the Wolff potential term in \eqref{Equ:leave-potential}, 
we will also need an improved \eqref{Equ:potential-estimate}. For this purpose we first 
consider the convergent infinite series
$$
\sum_{i=1}^\infty \mu(\Omega_i) \leq 3 \mu(B(0, 1)) < \infty
$$
and use Paul du Bois-Reymond Theorem \cite[(5) Page 40]{Brom} (cf. \cite{Rey}) 
to find a sequence $\zeta_{i}\rightarrow0^+$ as $i\rightarrow\infty$ such that 
\[
\sum_{i=1}^{\infty}\frac{1}{\zeta_{i}}\mu(\Omega_i) < \infty.
\]
for all $y\in A_{0, 1}$. 
From the similar argument as in the proof of \eqref{Equ:potential-estimate}, we have,
\begin{equation}\label{Equ:improved-potential-estimate}
cap_{n}(\{y\in B_{j}^{i}: W_{1,n}^{\mu}(y, \frac 12 |y|)>\frac{\zeta_i^\frac1{n-1}}{C_{1}}\log\frac{1}{|y|}\},2B_{j}^{i})
\leq\frac{C\frac 1{\zeta_i} \mu(2B_{j}^{i})}{i^{n-1}}.
\end{equation}
Let 
$$
E_{ij} = \{y\in B_{j}^{i}: W_{1,n}^{\mu}(y, \frac 12 |y|)>\frac{\zeta_i^\frac1{n-1}}{C_{1}}\log\frac{1}{|y|}\} \bigcap \omega_i, 
\ E_i = \bigcup_j E_{ij}, \text{ and }E = \bigcup_i E_i.
$$
Then \eqref{Equ:improved-potential-estimate} implies that
$$
\sum_{i=1}^\infty i^{n-1} cap_n(E\bigcap\omega_i, \Omega_i) \leq \sum_i \frac 1{\zeta_i}\mu(\Omega_i) < \infty,
$$
which says that $E$ is $n$-thin and
\begin{equation}\label{Equ:outside-bad-set}
\lim_{y\notin E \text{ and } y\to 0}  \frac {W_{1, n}^\mu(y, \frac 12 |y|)}{\log\frac 1{|y|}} = 0.
\end{equation}
Combining 
\begin{equation}\label{Equ:left-side}
\frac {w(y)}{\log\frac 1 {|y|}} \geq \gamma^-
\end{equation}
with \eqref{Equ:leave-potential} and \eqref{Equ:outside-bad-set}, we finally arrive at
$$
\lim_{y\notin E \text{ and } y\to 0} \frac {w(y)}{\log\frac 1{|y|}} = \gamma^- = \liminf_{y\to 0}  \frac {w(y)}{\log\frac 1{|y|}}.
$$
At last we will prove that, if $(B(0,2)\backslash\{0\},e^{2w}|dx|^{2})$
is complete at the origin, then $m\ge1$. Since
\[
\lim_{x\notin E,x\to0}\frac{w(x)}{\log\frac{1}{|x|}}=m,
\]
 if we can find a ray $P$ starting from $0$, such that $P\cap E\cap B(0,r_{0})=\emptyset,0<r_{0}<2$,
then from the completeness, for any $\varepsilon>0$
\[
+\infty=\int_{P\cap B(0,r_{0})}e^{w}dr\le\int_{0}^{r_{0}}\frac{1}{r^{m+\varepsilon}}dr.
\]
So we know that $m+\varepsilon\ge1$, which implies that $m\ge1$.
The question is reduced to finding such ray $P$ which has no intersection
with the thin set $E$, at least in a small ball $B(0,r_{0})$.

Define the projection map
\begin{align*}
Pr:\omega(0,0) & \mapsto\partial B(0,1),\\
(r,\theta) & \to(1,\theta).
\end{align*}
It is obvious a Lipschitz map, with Lipschitz constant $2$. From
the conformal invariance property of $n$-capacity, we know 
\[
cap_{n}(E\cap\omega(0,i),\Omega(0,i))=cap_{n}((2^{i}E)\cap\omega(0,0),\Omega(0,0)).
\]
From the monotonicity property of capacity with respect to a Lipschitz
map, Theorem 5.2.1 of \cite{AH96}, we know 
\[
cap_{n}((2^{i}E)\cap\omega(0,0),\Omega(0,0))\ge cap_{n}(Pr((2^{i}E)\cap\omega(0,0)),\Omega(0,0)).
\]
 From the thin property of $E$, we know 
\[
\sum_{i}i^{n-1}cap_{n}(Pr((2^{i}E)\cap\omega(0,0)),\Omega(0,0))<+\infty.
\]
So 
\begin{align*}
 & \lim_{i_{0}\to+\infty}cap_{n}(\cup_{i\ge i_{0}}Pr(2^{i}E\cap\omega(0,0)),\Omega(0,0))\\
\le & \lim_{i_{0}\to+\infty}\sum_{i\ge i_{0}}cap_{n}(Pr(2^{i}E\cap\omega(0,0)),\Omega(0,0))\\
= & 0.
\end{align*}
Since $cap_{n}(\partial B(0,1),\Omega(0,0))$ is a positive number
depending on $n$, so we can find a ray $P$ and $r_{0}>0$, such that
$P\cap E\cap B(0,r_{0})=\emptyset$.
Thus the proof of Theorem \ref{Thm:AH-1} is completed.
\end{proof}

%Discussion of the new thin notion. 

\section{Higher dimensional Taliaferro's estimates}\label{Sec:taliaferro}

Let us start with Taliaferro's estimates in 2 dimensions.

\begin{theorem*} (\cite[Theorem 2.1]{Tal-2}) Suppose that $u$ is $C^2$ positive solution to
$$
0\leq -\Delta u \leq f(u)
$$
in a punctured neighborhood of the origin in $\mathbb{R}^2$, where $f: (0, \infty)\to (0, \infty)$ is a continuous function such that 
$$
\log f(t) = O(t) \text{ as $t\to \infty$}.
$$
Then, either $u$ has a $C^1$ extension to the origin or 
\begin{equation}\label{Equ:good limit}
\lim_{x\to 0} \frac {u(x)}{\log \frac 1{|x|}} = m
\end{equation}
for some finite positive number $m$.
\end{theorem*}
This can be viewed as the improvement of \cite[Theorem 1.3]{AH}, having no thin subset where 
the asymptotic behavior may differ from \eqref{Equ:good limit}. Our next goal is to establish the higher dimensional analogue of 
\cite[Theorem 2.1]{Tal-2} as follows:

\begin{theorem} \label{Thm:main theorem-2} Let $w \in C^2(B(0, 2)\setminus \{0\})$ be nonnegative and satisfy
\begin{equation}\label{Equ:nonlinearity}
 -\Delta_n w = f(x, w, \nabla w)
\end{equation}
in a punctured neighborhood of the origin in $\mathbb{R}^n$ and that
$$
\lim_{x\to 0} w(x) = +\infty, 
$$ 
where $f$ is a nonnegative function satisfying
\begin{equation}\label{Equ:growth condition}
0 \leq f (x, w, \nabla w) \leq C|\nabla w|^{n-2}e^{2w}
\end{equation}
for some fixed constant $C$. Then 
\begin{equation}\label{Equ:theorem 4.1}
\lim_{|x|\rightarrow 0}\frac{w(x)}{\log\frac{1}{|x|}}= m \geq 0
\end{equation}
and
\[
w(x)\geq m\log\frac{1}{|x|} \text{ for $x\in B(0, 1)\setminus \{0\}$}.
\]
Moreover, if $e^{2w}|dx|^2$ is complete and non-compact at the origin, then $m\geq1$. 
\end{theorem}

\begin{remark} \label{Rem:critical growth}
We would like to make a remark that the growth condition \eqref{Equ:growth condition} can be replaced by
\begin{equation}\label{Equ:growth condition-1}
0 \leq f(x, w, \nabla w) \leq C |\nabla w|^p e^{\alpha w}
\end{equation}
for any $p\in (0, n)$ and $\alpha >0$. This can be seen from \eqref{Equ:critical} in the proof of Lemma \ref{Lem:no thin set} and
\eqref{Equ:when s small} in the proof of Theorem \ref{Thm:main theorem-2}.
\end{remark} 

%%%%%%%%%%%%%%%%%%%%%%%%%%%%%%%%%
 
\subsection{The extension of Brezis-Merle inequality in higher dimensions}

From Theorem \ref{Thm:B-V}, we know, for some $\beta \ge 0$
$$ -\Delta_n w=\beta \delta_0+f(x,w,\nabla w)=\mu,
$$ where $f(x,w,\nabla w)\in L^1(B(0,1))$.  The key analytic tool to remove the possibility of concentrating for solutions to $n$-Laplace equations like \eqref{Equ:intro-n-Laplace}
and \eqref{Equ:nonlinearity} with the critical growth condition \eqref{Equ:growth condition} or more generally 
\eqref{Equ:growth condition-1} is the higher dimensional analogue 
of the borderline Sobolev inequality established by Brezis and Merle in 2 dimensions in \cite[Theorem 1]{BM}, 
like Adams-Moser-Trudinger inequalities
(please see \cite{FM, Io} and references therein). Our approach is different from \cite{Io} but the result is similar to that in \cite{Io}. 
To extend \cite[Theorem 1]{BM} to general dimensions, we recall the Wolff potential 
\[
W_{1,n}^{\mu}(x,r)=\int_{0}^{r}\mu(B(x,t))^{\frac{1}{n-1}}\frac{dt}{t}
\]
associated with a Radon measure $\mu$, and a Radon measure $\mu_f$ that is induced from a function $f\in L^1(\Omega)$ 
\[
\mu_{f}(U)=\int_{U\cap\Omega}fdx.
\]

\begin{proposition}\label{Pro:brezis-merle} Let $\Omega\subset\mathbb{R}^{n}$
be a bounded domain with the diameter $D$. And let $f\in L^{1}(\Omega)$
be nonnegative. Then, for $\delta\in (0, 1)$, 
\begin{equation}\label{Equ:brezis-merle}
\int_{\Omega} {\rm exp}( {\frac{n(1-\delta)W_{1,n}^{\mu_{f}}(x,D)}{\|f\|_{L^{1}(\Omega)}^\frac 1{n-1}}})dx\leq 
\frac {c(n)2^{2n+1} |B(0, D)|}{ \delta^{n+1}} + 2^n |\Omega|. 
\end{equation}
\end{proposition}

\begin{proof}
The proof is more or less standard in harmonic analysis. For the convenience of readers, we present a proof here. 
To start, we let $p>n-1$ and $\alpha^{p}=\mu_f(B(x,D))=\|f\|_{L^{1}(\Omega)} \leq1$.
Then 
\begin{align*}
W_{1,n}^{\mu_f}(x,D) & \leq\int_{0}^{D}\mu_f(B(x,t))^{\frac{1}{p}}\frac{dt}{t}\\
 & =\mu(B(x,t))^{\frac{1}{p}}\log t|_{0}^{D}+\int_{0}^{D}\log\frac{1}{t}d\mu(B(x,t))^{\frac{1}{p}}.
\end{align*}
Let 
\[
Mf(x)=\sup_{t>0}\frac{1}{|B(x,t)|}\int_{B(x,t)\cap\Omega}f(y)dy = \sup_{t >0} \frac {\mu(B(x, t))}{|B(x, t)|}
\]
be the Hardy-Littlewood maximal function of $f$. Hence
$$
\mu(B(x,t)) \leq Mf(x) |B(0, t)| =  nw_{n-1}t^nMf(x)
$$
almost everywhere, that is to say,  
$$
\mu(B(x,t))^{\frac{1}{p}}\log t|_{0}^{D}=\alpha\log D
$$
almost everywhere. Therefore, by Jensen's inequality
\begin{align*}
\exp(W_{1,n}^{\mu}(x,D)) & \leq D^{\alpha}\int_{0}^{D}\frac{1}{t^{\alpha}}\frac{1}{\alpha}d\mu(B(x,t))^{\frac{1}{p}}\\
 & \leq D^{\alpha}(\frac{1}{\alpha}\frac{1}{t^{\alpha}}\mu(B(x,t))^{\frac{1}{p}}|_{0}^{D}
 +\int_{0}^{D}\mu(B(x,t))^{\frac{1}{p}}\frac{1}{t^{\alpha+1}}dt).
\end{align*}
If $\alpha<\frac{n}{p}$, then 
\begin{align*}
\exp(W_{1,n}^{\mu}(x,D)) & \leq D^{\alpha}(D^{-\alpha}+ \frac 1{\frac np - \alpha} (nw_{n-1})^\frac 1pMf(x)^{\frac{1}{p}}D^{\frac{n}{p}-\alpha})\\
 & =1+ \frac p{n - \alpha p}  |B(0, D)|^\frac 1p Mf(x)^{\frac{1}{p}}.
\end{align*}
So we have, for $\lambda \geq 2$,  
\begin{align*}
|\{x\in\Omega:\exp(W_{1,n}^{\mu}(x,D)) \geq\lambda\}| & \leq |\{x\in\Omega: Mf(x) \geq \frac { (n - \alpha p)^p\lambda^p}{2^p p^p
|B(0, D)|}  \}| \\
& \leq\frac{c(n)2^{p} p^p |B(0, D)|\|f\|_{L^{1}}}{ (n- \alpha p)^p \lambda^{p}},
\end{align*}
thanks to the weak type Hardy-Littlewood maximal inequality. For $0<q<p$,
\begin{align}
\int_{\Omega}\exp(qW_{1,n}^{\mu}(x,D))dx & =\int_{0}^{+\infty}|\{x\in \Omega:\exp(W_{1,n}^{\mu}(x,D))\geq t^{\frac{1}{q}}\}|dt\nonumber \\
 & \leq\int_{2^{q}}^{+\infty}\frac{c(n)2^{p} p^p |B(0, D)| \, \|f\|_{L^{1}(\Omega)}}{ (n - \alpha p)^p t^{\frac{p}{q}}}dt+\int_{0}^{2^{q}}|
 \Omega|dt\nonumber \\ & \leq\frac{c(n)q\, 2^{q} p^p}{(p-q)(n- \alpha p)^p} |B(0, D) \, |\|f\|_{L^{1}(\Omega)} 
 + 2^q|\Omega|.\label{Equ:q norm estimate}
\end{align}
Now consider $p=n(1 -\frac \delta 2)$, $q = n(1- \delta)$, $\delta\in (0, 1)$, and $\alpha = \|f\|_{L^{1}(\Omega)}=1$ 
(otherwise one may consider $\bar f = \frac f{\|f\|_{L^1(\Omega)}}$ instead). Then, from \eqref{Equ:q norm estimate}, we have
\[
\int_{\Omega} {\rm exp}({n(1-\delta) \frac {W_{1,n}^{\mu}(x,D)}{\|f\|_{L^1(\Omega)}^\frac 1{n-1}}}) dx
\leq \frac{c(n)2^{2n+1}|B(0, D)| }{\delta^{n+1}} + 2^n |\Omega|.
\]
This finishes the proof. 
\end{proof}

%%%%%%%%%%%%%%%%%%%%%%%%%%%%%%%%%

\subsection{The uniform bound for the quotients}
In contrast to the proof of Theorem \ref{Thm:AH-1} in the previous subsection, we will be able to show, based on the growth 
condition \eqref{Equ:growth condition} and Proposition \ref{Pro:brezis-merle},  the quotient $\frac {w(x)}{\log\frac 1{|x|}}$ is 
bounded: the analogue of \cite[Theorem 2.3]{Tal-2}.

\begin{lemma}\label{Lem:no thin set} Assume the same assumptions as in Theorem \ref{Thm:main theorem-2}. Then
the quotient
$$
\frac {w(x)}{\log\frac 1{|x|}}
$$
is uniformly bounded in the punctured ball $B(0, 1)\setminus\{0\}$.
\end{lemma} 
\begin{proof}

We prove Lemma \ref{Lem:no thin set} by contradiction. Assume otherwise,
there is a sequence $\{x_{k}\}$ inside the punctured ball such that
\[
\frac{w(x_{k})}{\log\frac{1}{|x_{k}|}}\to\infty\text{ as \ensuremath{|x_{k}|\to0}}.
\]
One may consider the blow-up sequence 
\[
v_{k}(\xi)=w(x_{k}+\frac{|x_{k}|}{4}\xi)\text{ for \ensuremath{\xi\in B(0,2)}}
\]
and calculate 
\begin{align}
-\Delta_{n}^{\xi}v_{k}=-(\frac{|x_{k}|}{4})^{n}\Delta_{n}^{x}w(x_{k}+\frac{|x_{k}|}{4}\xi)
=g_{k}(\xi) & \leq C|x_{k}|^{2}|\nabla^{\xi}v_{k}|^{n-2}e^{2v_{k}}\text{ for \ensuremath{\xi\in B(0,2)}},\nonumber\\
\int_{B(0,2)}g_{k}(\xi)d\xi  & = 
\int_{B(x_{k},\frac{|x_{k}|}{2})}g(x)dx\to0\text{ as \ensuremath{k\to\infty}}, \label{Equ:small-l-1}
\end{align}
where $-\Delta_{n}w=g+\beta\delta_{0}$ and $g\in L_{\text{loc}}^{1}(B(0,2))$
according to \cite[Proposition 1.1]{B-V}. We will argue in the similar
way to that in \cite{Tal-2}. We combine the non-linear potential
theory \cite[Theorem 1.6]{KM} with Lemma \ref{Lem:inf estimate}.
For convenience, let us denote 
\[
\lambda_{k}=\log\frac{1}{|x_{k}|}\to\infty\text{ as \ensuremath{k\to\infty}}.
\]
Then it is implied from \cite[Theorem 1.6]{KM} and Lemma \ref{Lem:inf estimate}
that 
\begin{align}
\frac{1}{\lambda_{k}} & W_{1,n}^{\mu_{g_{k}}}(0,2)\to\infty\label{Equ:going infinity}\\
g_{k}(\xi) & \leq C|x_{k}|^{2}|\nabla^{\xi}v_{k}|^{n-2}e^{C_{1}\lambda_{k}+C_{2}W_{1,n}^{\mu_{g_{k}}}(\xi,2)}\text{ for \ensuremath{\xi\in B(0,1)}}.\label{Equ:KM}
\end{align}
Here $\mu_{g_{k}}$ is a measure such that $\mu_{g_{k}}(E)=\int_{E}g_{k}d\xi,E\subset B(0,2)$.
A very important observation is that, when dealing with competing
terms like $\lambda_{k}$ and $W_{1,n}^{\mu_{g_{k}}}(0,2)$, for 
\[
\Omega_{k}=\{\xi\in B(0,1):W_{1,n}^{\mu_{g_{k}}}(\xi,2)\geq\lambda_{k}\}
\]
we have 
\begin{align}
\begin{aligned}\int\end{aligned}
_{\Omega_{k}}|g_{k}|^{\frac{n-1}{n-2}}d\xi & \leq C|x_{k}|^{\frac{2(n-1)}{n-2}}\int_{\Omega_{k}}|\nabla^{\xi}v_{k}|^{n-1}e^{\frac{2(n-1)}{n-2}v_{k}}d\xi\nonumber \\
 & \leq C|x_{k}|^{\frac{2(n-1)}{n-2}}\int_{\Omega_{k}}|\nabla^{\xi}v_{k}|^{n-1}e^{\frac{2(n-1)}{n-2}(C_{1}(n)\inf_{B(0,1)}v_{k}+C_{2}(n)W_{1,n}^{\mu_{g_{k}}}(\xi,2))}d\xi\nonumber \\
 & \le C|x_{k}|^{\frac{2(n-1)}{n-2}}\int_{B(0,1)}|\nabla^{\xi}v_{k}|^{n-1}e^{C_{3}(n)W_{1,n}^{\mu_{g_{k}}}(\xi,2)}d\xi\label{Equ:critical}\\
 & \leq C|x_{k}|^{\frac{2(n-1)}{n-2}}(\int_{B(0,1)}|\nabla^{\xi}v_{k}|^{n-\frac{1}{2}}d\xi)^{\frac{2n-2}{2n-1}}(\int_{B_{2}(0)}e^{C_{4}(n)W_{1,n}^{\mu_{g_{k}}}(\xi,2)}d\xi)^{\frac{1}{2n-1}}\nonumber\\
 & \leq C|x_{k}|^{\frac{2(n-1)}{n-2}-\frac{n-1}{2n-1}}(\int_{B(x_{k},\frac{|x_{k}|}{2})}|\nabla^{x}w|^{n-\frac{1}{2}}dx)^{\frac{2n-2}{2n-1}}(\int_{B_{2}(0)}e^{C_{4}(n)W_{1,n}^{\mu_{g_{k}}}(\xi,2)}d\xi)^{\frac{1}{2n-1}} \nonumber\\
 & \leq C.\nonumber
\end{align}

Make a note that $\frac{2(n-1)}{n-2}-\frac{n-1}{2n-1}>1$. The last
step in the above inequalities relies on Proposition \ref{Pro:brezis-merle}
and the $L^{p}$-gradient estimates for $n$-superharmonic functions
for any $p<n$ (for example by Theorem \ref{Thm:B-V}). This implies that 
\[
\mu_{g_{k}}(B(0,t)\cap\Omega_{k})\leq Ct^{\frac{n}{n-1}}
\]
for some positive constant $C>0$. Observe that 
\[
\mu_{g_{k}}(B(0,t))^{\frac{1}{n-1}}\leq\mu_{g_{k}}(B(0,t)\cap\Omega_{k})^{\frac{1}{n-1}}+\mu_{g_{k}}(B(0,t)\setminus\Omega_{k})^{\frac{1}{n-1}}
\]
which implies 
\begin{equation}
W_{1,n}^{\mu_{g_{k}}}(0,2)\leq C+C\int_{0}^{2}\mu_{g_{k}}(B(0,t)\setminus\Omega_{k})^{\frac{1}{n-1}}\frac{dt}{t}.\label{Equ:basic inequality}
\end{equation}
To estimate the second term on the right side the above equation,
one notices that, for $\xi\in B(0,1)\setminus\Omega_{k}$, 
\[
g_{k}(\xi)\leq C|x_{k}|^{2}|\nabla^{\xi}v_{k}|^{n-2}e^{C_{2}(n)\lambda_{k}+C_{3}(n)W_{1,n}^{\mu_{g_{k}}}(\xi,2)}\leq C|x_{k}|^{2}|\nabla^{\xi}v_{k}|^{n-2}e^{C_{5}(n)\lambda_{k}}
\]
from \eqref{Equ:KM}. Therefore 
\begin{align*}
\int_{B(0,t)\backslash\Omega_{k}}g_{k}(\xi)d\xi\leq & C\int_{B(0,t)\backslash\Omega_{k}}|x_{k}|^{2}|\nabla^{\xi}v_{k}|^{n-2}e^{C_{5}(n)\lambda_{k}}d\xi\\
\leq & C|x_{k}|^{2-\frac{n-2}{n-1}}e^{C_{5}(n)\lambda_{k}}\int_{B(0,t)\backslash\Omega_{k}}|x_{k}|^{\frac{n-2}{n-1}}|\nabla^{\xi}v_{k}|^{n-2}d\xi\\
\leq & C|x_{k}|^{2-\frac{n-2}{n-1}}e^{C_{5}(n)\lambda_{k}}t^{\frac{n}{n-1}}(\int_{B(0,t)\backslash\Omega_{k}}(|x_{k}|^{\frac{n-2}{n-1}}|\nabla^{\xi}v_{k}|^{n-2})^{\frac{n-1}{n-2}}d\xi)^{\frac{n-2}{n-1}}\\
\leq & C|x_{k}|^{2-\frac{n-2}{n-1}}e^{C_{5}(n)\lambda_{k}}t^{\frac{n}{n-1}}(\int_{B^{x}(0,1)}|\nabla^x w|^{n-1}dx)^{\frac{n-2}{n-1}}.
\end{align*}
We now calculate separately, for $\rho_{k}$ to be fixed next, 
\begin{align*}
\begin{aligned}\int\end{aligned}
_{0}^{2}\mu_{g_{k}}(B(0,t)\setminus\Omega_{k})^{\frac{1}{n-1}}\frac{dt}{t} & =\int_{0}^{\rho_{k}}\mu_{g_{k}}(B(0,t)\setminus\Omega_{k})^{\frac{1}{n-1}}\frac{dt}{t}+\int_{\rho_{k}}^{2}\mu_{g_{k}}(B(0,t)\setminus\Omega_{k})^{\frac{1}{n-1}}\frac{dt}{t}\\
 & \leq(n-1)^{2}C|x_{k}|^{2}e^{C_{5}(n)\lambda_{k}}\rho_{k}^{\frac{1}{(n-1)^{2}}}+C\log\frac{1}{\rho_{k}}+C.
\end{align*}
Let us fix 
\[
\rho_{k}=e^{-(n-1)^{2}C_{5}(n)\lambda_{k}}\in(0,2).
\]
We thus get 
\[
\int_{0}^{2}\mu_{g_{k}}(B(0,t)\setminus\Omega_{k})^{\frac{1}{n-1}}\frac{dt}{t}\leq C+C\lambda_{k},
\]
which contradicts with \eqref{Equ:going infinity} in the light of
\eqref{Equ:basic inequality}. So Lemma \ref{Lem:no thin set} is
proved. 
\end{proof}

%%%%%%%%%%%%%%%%%%%%%%%%%%%%%%%%%%%

\subsection{The proof of Theorem \ref{Thm:main theorem-2}}
Lemma \ref{Lem:no thin set} enables us to proceed with blow-down argument without going through
Sections \ref{Subsect:step-1-AH} and \ref{Subsect:step-2-AH}. We are now ready to prove Theorem \ref{Thm:main theorem-2}.

\begin{proof}[The proof of Theorem \ref{Thm:main theorem-2}]
We again consider the blow-down 
\[
w_{r}(\xi)=\frac{w(r\xi)}{\log\frac{1}{r}}
\]
and calculate that, from Lemma \ref{Lem:no thin set}, 
\[
|w_{r}(\xi)|\leq C\frac{\log\frac{1}{r}+|\log\frac{1}{|\xi|}|}{\log\frac{1}{r}}\leq2C\text{ for all \ensuremath{\xi\in A_{r,\frac{1}{r}}=\{\xi\in\mathbb{R}^{n}:|\xi|\in(r,\frac{1}{r})\}}}.
\]
From here, similar to the approach of the proof of Theorem \ref{Thm:AH-1}
in the previous section, one may complete the proof of Theorem
\ref{Thm:main theorem-2}. To do so, we continue to use the notation
$$
\gamma^- = \liminf_{x\to 0} \frac {w(x)}{\log\frac 1{|x|}}
$$
as in \eqref{Equ:weak-limit} in Section \ref{Subsect:proof of Theorem 3.1}.
\\

First, as in the previous section, one may prove that on $\mathbb{R}^{n}\backslash\{0\}$, $w_{r}(\xi)$
converges to $\gamma^{-}$ in $W_{loc}^{1,n}$ weakly and $W_{loc}^{1,p}$ strongly for any $p< n$,
which implies that $w_r(\xi)$ converges to $\gamma^-$ pointwisely almost everywhere. This heavily relies on the uniqueness
of sequential blow-down limits established in the proof Theorem \ref{Thm:AH-1} in Section \ref{Subsect:proof of Theorem 3.1}.  
Hence we want to improve from here that $\frac {w(x)}{\log\frac 1{|x|}}$ converges to $\gamma^{-}$ pointwisely as
we did in the proof Theorem \ref{Thm:AH-1} in Section \ref{Subsect:proof of Theorem 3.1}. In the light of \eqref{Equ:left-side} and 
\eqref{Equ:leave-potential}, we need to show \eqref{Equ:outside-bad-set} with no thin set $E$ excluded, i.e.
\begin{equation}\label{Equ:no thin set-4}
\lim_{x\to 0}\frac{W_{1, n}^\mu(x, \frac 12|x|)}{\log\frac 1{|x|}} = 0.
\end{equation}
To prove \eqref{Equ:no thin set-4}, we recall that
$$
W_{1, n}^\mu(x, \frac 12 |x|) = \int_0^{\frac 12 |x|} \mu (B(x, s))^\frac 1{n-1}\frac {ds}s,
$$
where
$$
\mu(B(x, s)) = \int_{B(x, s)} f(x, w, \nabla w) dx \leq C \int_{B(x, s)} |\nabla w|^{n-2} e^{2w} dx.
$$
From \eqref{Equ:small-l-1} and \cite[Proposition 1.1]{B-V}, we know that  $\mu(B(0, 1)) < \infty$ and that 
$\mu(B(x, s))\to 0$ as $s\to 0$ for $s\leq \frac 12|x|$.  
But that is not enough, particularly when $s$ is very small in calculating the Wolff potential. Therefore we 
recall \cite[Theorem 1.6]{HK}
$$
w(x) \leq C_2 \inf_{B(x, \frac 14|x|)}w + C_3W_{1, n}^\mu (x, \frac 12|x|)
$$
and estimate, 
\begin{align*}
\mu(B(x, s)) \leq C |x|^{-2C_2(\gamma^- + 1)} \int_{B(x, s)} |\nabla w|^{n-2}e^{2C_3W_{1, n}^\mu(x, \frac 12|x|)} dx.
\end{align*}
Applyin H\"{o}lder inequality, we have
$$
\mu(B(x, s)) \leq C |x|^{-2C_2 (\gamma^- + 1)} (\int_{B(x, s)}|\nabla w|^{n-1}dx)^\frac {n-2}{n-1} (\int_{B(x, s)} e^{2(n-1)
C_3W_{1, n}^\mu(x, \frac 12|x|)} dx )^\frac 1{n-1} 
$$
Then, we use Proposition \ref{Pro:brezis-merle} and derive
$$
\mu(B(x, s)) \leq C |x|^{-2C_2(\gamma^- +1)} (\int_{B(x, s)}|\nabla w|^{n-1}dx)^\frac {n-2}{n-1}
$$
Finally, we use $L^p$ bound for the gradient of the $n$-superharmonic function $w$ for $p = n -\frac 12 < n$ and get
\begin{align}
\mu(B(x, s)) & \leq C|x|^{-2C_2 (\gamma^- +1)} (\int_{B(x, s)}|\nabla w|^{n-\frac 12}dx)^\frac {n-2}{n-\frac 12} 
s^\frac {n(n-2)}{2(n-1)(n-\frac 12)} \nonumber\\
& \leq C |x|^{-2C_2 (\gamma^- + 1)} s^\frac {n(n-2)}{2(n-1)(n-\frac 12)}, \label{Equ:when s small}
\end{align}
Here we are indifferent to constants except maybe those from \cite[Theorem 1.6]{HKM}. Therefore, 
going back to estimate the Wolff potential, we have
\begin{align*}
W_{1, n}^\mu(x, \frac 12|x|) & = \int_0^\rho \mu (B(x, s))^\frac 1{n-1} \frac {ds}s + \int_\rho^{\frac 12|x|}
 \mu (B(x, s))^\frac 1{n-1} \frac {ds}s \\
& \leq C |x|^{-\frac{2C_2 (\gamma^- + 1)}{n-1}} \rho^\frac {n(n-2)}{2(n-1)^2(n-\frac 12)} + o(1) \log \frac 1\rho,
\end{align*}
for the choice
$$
\rho = |x|^{\frac {4C_2(\gamma^-+1)(n-1)(n-\frac 12)}{n(n-2)}}
$$
and $o(1)$ is with respect to $x\to 0$, which implies \eqref{Equ:no thin set-4}. So \eqref{Equ:theorem 4.1} is 
established. It is then easily seen that
$$
\int_0^1 e^w dr = \infty
$$
implies $m\geq 1$ from \eqref{Equ:theorem 4.1}. Thus the proof is completed. 
\end{proof}

%%%%%%%%%%%%%%%%%%%%%%%%%%%%%%%%%%%%%%%%%%%%%%%%%%%%%%%%
%%%%%%%%%%%%%%%%%%%%%%%

\section{Locally conformally flat manifolds}\label{Sec:conformally flat}

In this section we are going to use the property of $n$-superharmonic functions to study the asymptotic behavior
at the end of a complete locally conformally flat manifold $(M^n, g)$. Based on the injectivity of the development maps of
\cite[Theorem 4.5]{SY}, in \cite[Theorem 1]{Zhu} and later in \cite{CH-1}, the following classification result was shown.

\begin{theorem*} (\cite{Zhu} \cite{CH-1}) Let $(M^n, g)$ be a complete conformally flat manifold of dimension $n\geq 3$ 
with nonnegative Ricci curvature. Then, exactly one of the following holds:
\begin{itemize}
\item $M$ is globally conformally equivalent to $\mathbb{R}^n$ with a conformal non-flat metric with nonnegative Ricci curvature;
\item $M$ is globally conformally equivalent to a spaceform of positive curvature, endowed with a conformal metric with nonnegative Ricci curvature;
\item $M$ is locally isometric to the cylinder $\mathbb{R}\times\mathbb{S}^{n-1}$;
\item $M$ is isometric to a complete flat manifold.
\end{itemize}
\end{theorem*}

We confine ourselves to the first case in the above classification theorem.  
Recall that, on $(\mathbb{R}^n, e^{2\phi}|dx|^2)$, in the light of \eqref{Equ:intro-n-Laplace},
$$
-\Delta_n \phi = \text{Ric}_g(\nabla^g\phi) |\nabla\phi|^{n-2}e^{2\phi},
$$
where $\text{Ric}_g(\nabla^g \phi)$ is the Ricci curvature of the conformal metric $g= e^{2\phi}|dx|^2$ in the $\nabla^g \phi$ direction. 
As a consequence of Theorem \ref{Thm:AH-1} and Theorem \ref{Thm:main theorem-2}, for a globally conformally flat manifold
$(\mathbb{R}^n, e^{2\phi}|dx|^2)$, we therefore are able to deduce the following:

\begin{theorem}\label{Thm:main theorem-g} Suppose that $(\mathbb{R}^{n},e^{2\phi}|dx|^{2})$
is complete with nonnegative Ricci ($n\geq3$), where $\phi$ is a smooth function. Then there is a subset
$E\subset\mathbb{R}^{n}$, which is $n$-thin
at infinity, such that 
\begin{equation}
\lim_{x\notin E\to\infty}\frac{\phi(x)}{\log\frac{1}{|x|}}=\liminf_{x\to\infty}\frac{\phi(x)}{\log\frac{1}{|x|}}=m\label{Equ:with thin set}
\end{equation}
and 
\begin{equation}
\phi(x)\geq m\log\frac{1}{|x|} - C \label{phi lower bdd}
\end{equation}
for some constant $C$, where  
\begin{equation}
m|m|^{n-2} =\frac 1{w_{n-1}} \int_{\mathbb{R}^{n}}\text{Ric}_{g}(\nabla^{g}\phi)|\nabla\phi|^{n-2}e^{2\phi}dx.\label{Equ:theorem 5.1-2}
\end{equation}
Moreover, 
\begin{itemize}
\item $m\in[0,1]$ and $m=0$ if and only if $g$ is flat, i.e. $\phi(x)$
is a constant function;
\item if $\text{Ric}_{g}$ is bounded in addition, then 
\begin{equation}
\lim_{x\to\infty}\frac{\phi(x)}{\log\frac{1}{|x|}}=\liminf_{x\to\infty}\frac{\phi(x)}{\log\frac{1}{|x|}}=m. \label{Equ:theorem 5.1-3}
\end{equation}
\end{itemize}
\end{theorem}
We remark that Theorem \ref{Thm:main theorem-g} should be compared with \cite{BKN, Cold, CZ}. In \cite{BKN} it was proved that, a complete
noncompact manifold $(M^n, \ g)$ satisfying
\begin{align*}
\text{Ric} & \geq 0\\
\text{vol}(B(0, r)) & \geq \gamma r^n \text{ for some $\gamma > \frac 12w_{n-1}$}\\
|\text{Rm} | & \leq Cr^{-2}
\end{align*}
and in addition, 
$$
\text{either } |\text{Rm}| = o(r^{-2}) \text{ or } \int_M |\text{Rm}|^\frac n2 dvol < \infty,
$$
is actually isometric to the Euclidean space. The assumption of $\gamma > \frac 12w_{n-1}$ is essential, in the light of Eguchi-Hanson
metrics. In \cite{Cold}, Colding proved remarkably that a complete manifold with nonnegative Ricci curvature is isometric to the Euclidean
space, if one tangent cone at infinity is the Euclidean space. 
In \cite{CZ}, on the other hand, it was proved, a complete noncompact conformally flat manifold with nonnegative Ricci and satisfying
$$
\frac 1{\text{vol}(B(x_0, r))}\int_{B(x_0, r)} R dvol = o(r^{-2})
$$
where the scalar curvature $R$ is bounded, is actually isometric to the Euclidean space. 
The comparison of Theorem \ref{Thm:main theorem-g} to the rigidity results in \cite{BKN, Cold, CZ} 
would be more direct if the intrinsic distance function $r$ on the manifold and $|x|^{1-m}$ 
in Euclidean space as the background metric are equivalent, which seems to require something 
stronger than \eqref{Equ:theorem 5.1-3}.
\\

\begin{proof}[The proof of Theorem \ref{Thm:main theorem-g}]
First we use the inversion to turn the asymptotic problem to be the one at around the origin as those 
studied in Theorem \ref{Thm:AH-1}and Theorem \ref{Thm:main theorem-2}. Let 
$$ 
w (y) = \phi(\frac y{|y|^2}) - 2\log |y|
$$
for $y\in \mathbb{R}^n\setminus \{0\}$. Then $g = e^{2\phi(x)}|dx|^{2}=e^{2w(y)} |dy|^{2}$.
Then from (\ref{Equ:intro-n-Laplace}) we know
\[
-\Delta_{n}^{y}w(y)=\text{Ric}_g(\nabla^g w)|\nabla w|^{n-2} e^{2w}.
\]
Because $g = e^{2w}|dy|^{2}$ is complete at the origin and its scalar curvature $R\geq0$, from \cite[Proposition 8.1]{CHY}, 
we know that 
\[
\lim_{y\rightarrow0}w(y)=+\infty.
\]
Hence, from Theorem \ref{Thm:B-V} and Theorem \ref{Thm:AH-1}, we know there are a number
$m_{1}\ge1$ and a set $E_{1}$, which is $n$-thin at $0$ such that,
\[
\lim_{y\notin E_{1},y\rightarrow0}\frac{w(y)}{\log\frac{1}{|y|}}=m_{1}
\]
and $w(y)\ge m_{1}\log\frac{1}{|y|} - C$. Now, translating these back to $\phi(x)$ through the inversion, we have
\begin{align}
\phi(x) \geq -m\log|x|- C & \text{ for any $|x|$ large} \label{Equ:|x| large}\\ 
\phi (x) \leq -m\log|x|+o(\log |x|) & \text{ for any $|x|$ large and outside of a set $E$}, \label{Equ:|x| large and thin}
\end{align}
where $m= 2 - m_{1} \leq 1$ and $E= \{x;\frac{x}{|x|^{2}}\in E_{1}\}$. Moreover, from Definition \ref{Def:n-thin}, 
we know $E$ is $n$-thin at infinity. So \eqref{Equ:with thin set} is proved. 
\\

If, in addition, Ricci curvature is bounded, then 
\[
\text{Ric}_g (\nabla^g w) |\nabla w|^{n-2}e^{2w} \leq C |\nabla^{y}w|^{n-2} e^{2w}
\]
and \eqref{Equ:theorem 5.1-3} follows from Theorem \ref{Thm:main theorem-2}. Assume \eqref{Equ:theorem 5.1-3} holds. Then it is
obvious that $m\in [0, 1]$. 
If (\ref{Equ:theorem 5.1-2}) holds, then it is obvious that $m\ge 0$, and if equality holds, then $\text{Ric}_g (\nabla^g w) |\nabla w|^{n-2}e^{2w} $ 
must be identically $0$, which implies that $\phi(x)$ is an $n$-harmonic function on $\mathbb{R}^n$, which is lower bounded by a constant from \eqref{phi lower bdd}. 
So $\phi$ has
to be a constant in this case due to \cite[Theorem 6.2 and Corollary 6.11]{HKM}.
\\

To finish the proof of Theorem \ref{Thm:main theorem-g}, it suffices to prove \eqref{Equ:theorem 5.1-2}. To do so, we are going to 
integrate 
\begin{equation}\label{Equ:integrate}
\int_\Omega \text{Ric}_g(\nabla^g \phi) |\nabla\phi|^{n-2} e^{2\phi} dx = \int_\Omega (-\Delta_n \phi)dx
=  - \int_{\partial\Omega} |\nabla\phi|^{n-2} \frac {\partial\phi}{\partial \vec{n}} dS_x.
\end{equation}
To avoid relying on sharp gradient estimates for $\phi$ on the boundary of any exhausting family of domains $\Omega$ 
in $\mathbb{R}^n$, we will work with chosen exhausting families of domains. Our construction of the exhausting families of 
domains is ingenious and turns out to be very natural and very desirable. Let us define, for $m\in \mathbb{R}$ and a positive small number
$\varepsilon$ and a positive large number $t$, 
\begin{align*}
G_{\varepsilon,t}^{+} (x) & =-(m+\varepsilon)\max\{\log|x|,0\}+t,\\
G_{\varepsilon,t}^{-}(x) & =-(m-\varepsilon)\max\{\log|x|,0\}-t.
\end{align*}
And let 
\begin{align*}
\Omega_{\varepsilon,t}^{+} & = \text{the connected component of $\{x: G_{\varepsilon,t}^{+}(x) > \phi(x)\}$ that includes the origin},\\
\Omega_{\varepsilon,t}^{-} & =\text{the connected component of $\{x: G_{\varepsilon,t}^{-}(x) < \phi(x)\}$ that includes the origin}.
\end{align*}

\begin{claim} For a fixed $\varepsilon > 0$, there is a sequence of positive number $t_k\to\infty$ such that 
the collection $\{\Omega_{\varepsilon, t_k}^+\}$ is an exhausting family of smooth and bounded domains for $\mathbb{R}^n$. 
Similarly, for a fixed $\varepsilon > 0$, there also exists a sequence of positive number $s_k\to\infty$ such that 
the collection $\{\Omega_{\varepsilon, s_k}^-\}$ is an exhausting family of smooth and bounded domains for $\mathbb{R}^n$. 
\end{claim}

\begin{proof}[Proof of Claim] Let us first consider $\Omega_{\varepsilon, t}^+$. Smoothness is not a problem, one can always 
perturb and get the smooth ones. From the definition, it is easily seen that, for any fixed $R$, 
$$
B(0, R)\subset \Omega_{\varepsilon, t}^+
$$
whenever $t$ is sufficiently large. Hence $\Omega_{\varepsilon, t}^+$ can exhaust the entire space. Meanwhile, for each fixed 
$\varepsilon$ and $t$, $\Omega_{\varepsilon, t}^+$ is bounded in the light of \eqref{Equ:|x| large}.
\\

Let us turn to $\Omega_{\varepsilon, t}^-$. The only issue different is the boundedness for $\Omega_{\varepsilon, t}^-$ when $\varepsilon$
and $t$ are arbitrarily fixed. It is easily seen that each $\Omega_{\varepsilon, t}^-\setminus E$ is bounded, because of 
\eqref{Equ:|x| large and thin}. Then $\Omega_{\varepsilon, t}^-$ is the connected component that includes the origin
and $n$-thin at infinity. Let $\gamma$ be a ray in Euclidean space starting from the origin. If $\Omega_{\varepsilon, t}^-$ is not bounded, then again from Theorem 5.2.1 of \cite{AH96}, we know for $i$ arbitrarily large,
$$cap(\Omega_{\varepsilon, t}^-\cap\omega (i,\infty),\Omega(i,\infty))\ge Ccap(\gamma\cap\omega (i,\infty),\Omega(i,\infty))\ge C(n)>0,$$
since the map from $\Omega_{\varepsilon, t}^-$ to $\gamma$ which preserves the radius is a Lipschitz map. 
Now we get a contradiction with the fact $E$ is $n$-thin at infinity. So the proof of this claim is finished.
\end{proof}
\begin{remark} In the above proof, Theorem 5.2.1 of \cite{AH96} can be replaced by the following lemma.
\begin{lemma}\label{Lem:gehring} (\cite[Lemma 1.4 page 212]{Red1} and \cite[Theorem 4]{G}) 
Let $K= (A, B)$ be a condenser in Euclidean $n$-space, where both $A$ and $B$ are connected. Assume that 

(1) $A$ is outside the unit ball, unbounded and includes a point on the unit sphere;

(2) $B$ includes the origin and has a point with length $L$.

\noindent
Then 
\begin{equation}\label{Equ:precise gehring estimate}
cap_n (B, A)  \geq \frac {c_n}{(\log ( 1+\frac 1L))^{n-1}}.
\end{equation}
Note that in our case $A=\Omega(i,\infty)^c$ or $\Omega(i,0)^c$ is unbounded. We will use this estimate in the next section.
\end{lemma}
\end{remark}

Now we return to the proof of Theorem \ref{Thm:main theorem-g}. On $\partial\Omega_{\varepsilon,t}^{+}$, we want 
\begin{equation}
|\nabla \phi|^{n-2}\frac{\partial \phi }{\partial\vec{n}}\geq|\nabla G_{\varepsilon,t}^{+}|^{n-2}\frac{\partial 
G_{\varepsilon,t}^{+}}{\partial\vec{n}}.\label{Equ:compare with G}
\end{equation}
This is because, in the normal direction at each point $x\in \partial\Omega_{\varepsilon,t}^{+}$,
$$
\frac{\partial \phi}{\partial\vec{n}} (x) \geq\frac{\partial G_{\varepsilon,t}^{+}}{\partial\vec{n}} (x)
$$
due to the definition of $\Omega_{\varepsilon, t}$. While, obviously, in the direction $\tau$ tangent to the boundary at each 
$x\in \partial\Omega_{\varepsilon,t}^{+}$,
$$
\frac{\partial \phi}{\partial\vec{\tau}} (x) = \frac{\partial G_{\varepsilon,t}^{+}}{\partial\vec{\tau}}(x).
$$
Therefore
\begin{itemize}
\item if $\frac{\partial G_{\varepsilon,t}^{+}}{\partial\vec{n}} (x) \geq0$, then we have 
$|\nabla \phi (x)|\geq |\nabla G_{\varepsilon,t}^{+}(x)|$ and \eqref{Equ:compare with G} holds; 
\item if $\frac{\partial G_{\varepsilon,t}^{+}}{\partial\vec{n}} (x) <0$
and $\frac{\partial \phi}{\partial\vec{n}} (x) \geq0$, \eqref{Equ:compare with G} trivially holds;
\item  if $\frac{\partial G_{\varepsilon,t}^{+}}{\partial\vec{n}} (x) \leq\frac{\partial\phi}{\partial\vec{n}} (x) <0$,
then $|\nabla \phi(x)|\leq|\nabla G_{\varepsilon,t}^{+}(x)|$ and still \eqref{Equ:compare with G} holds. 
\end{itemize}
So \eqref{Equ:compare with G} always holds as desired. Therefore, continuing from \eqref{Equ:integrate}, 
\begin{align*}
 & \int_{\Omega_{\varepsilon,t}^{+}}\text{Ric}_{g}(\nabla^g \phi)|\nabla \phi|^{n-2} e^{2\phi} dx\\
= & - \int_{\partial\Omega_{\varepsilon,t}^{+}}|\nabla \phi|^{n-2}\frac{\partial \phi }{\partial\vec{n}} dS\\
\leq & -\int_{\partial\Omega_{\varepsilon,t}^{+}}|\nabla G_{\varepsilon,t}^{+}|^{n-2}\frac{\partial G_{\varepsilon,t}^{+}}
{\partial\vec{n}}dS\\
= & (m+\varepsilon) |m+\varepsilon|^{n-2} w_{n-1}.
\end{align*}
Here in the last step, we use the fact that
$$
-\int_{\partial\Omega_{\varepsilon,t}^{+}}|\nabla \tilde G_{\varepsilon,t}^{+}|^{n-2}\frac{\partial \tilde G_{\varepsilon,t}^{+}}
{\partial\vec{n}}dS = \int_{\Omega_{\varepsilon, t}^+} (-\Delta \tilde G_{\varepsilon, t}^+)dx = (m+\varepsilon)|m+\varepsilon|^{n-2}
w_{n-1},
$$
for $t$ very large, where
$$
\tilde G_{\varepsilon, t}^+ =- (m+\varepsilon) \log|x| + t
$$
which agrees with $G_{\varepsilon, t}^+$ outside the unit ball.
\\

Similarly, using $G_{\varepsilon,t}^{-}$ and $\Omega_{\varepsilon, t}^-$, we have
\[
\int_{\Omega_{\varepsilon,t}^{-}} \text{Ric}_{g}(\nabla^g \phi) |\nabla \phi|^{n-2} e^{2\phi} dx
\geq (m-\varepsilon) |m-\varepsilon|^{n-2} w_{n-1}.
\]
Thus, by the exhaustion property of the chosen families of domains, \eqref{Equ:theorem 5.1-2} follows. The proof 
of Theorem \ref{Thm:main theorem-g} is completed.
\end{proof}

%%%%%%%%%%%%%%%%%%%%%%%%%%%%%%%%%%%%%%%%%%%%%%%%%%
%%%%%%%%%%%%%%%%%%%%%%%

\section{Hypersurfaces in hyperbolic space}\label{Sec:hypersurfaces}

In this section we want to use Theorem \ref{Thm:AH-1}
and Theorem \ref{Thm:main theorem-2} to study the asymptotic end
structure of embedded hypersurfaces in hyperbolic space with nonnegative
Ricci. Our work here is inspired by and improves the results in \cite{AlCu,AlCu2}. 
In the light of \cite[Main Theorem]{BMQ-r}, in this paper, we focus on the study of end 
structure at infinity for these hypersurfaces in hyperbolic space with nonnegative Ricci 
and one single end. We refer readers to Section \ref{Subsect:hypersurfaces} for a very 
brief introduction of complete and globally strictly convex hypersurfaces in hyperbolic space
(cf. \cite{AlCu, AlCu2, BMQ-s, BMQ-r}). For convenience of readers, we first remind us what is Busemann 
coordinates in hyperbolic space. We start with half space model for hyperbolic space
$$
R^{n+1}_+ = \{(x_1, x_2, \cdots, x_n, x_{n+1}): (x_1, x_2, \cdots, x_n)\in \mathbb{R}^n 
\text{ and } x_{n+1} > 0\}
$$
with the hyperbolic metric 
$$
g_{\mathbb{H}} = \frac {|dx|^2 + |dx_{n+1}|^2}{x_{n+1}^2}.
$$
We use the notation that $\partial_\infty\mathbb{H}^{n+1} = \mathbb{R}^n\bigcup \{p_\infty\}$
in this half space model. A vertical graph in hyperbolic space is the hypersurface given by 
$$
\phi (x) = (x, f(x)): \Omega\to\mathbb{R}^{n+1}_+, x=(x_1,\cdots,x_n)
$$
for a function
$$
x_{n+1}= f (x): \Omega\subset\mathbb{R}^n\to \mathbb{R}_+=\{s\in \mathbb{R}: s>0\}.
$$
The Busemann coordinates is $(x, \rho)\in \mathbb{R}^{n}\times \mathbb{R}$ such that 
$$
\rho = \log x_{n+1}.
$$
In this coordinates 
$$
g_{\mathbb{H}} = e^{-2\rho}|dx|^2 + d\rho^2.
$$ 
Therefore the height function for a vertical graph in Busemann coordinates is
$$
\rho (x) = \log f (x): \Omega \to \mathbb{R}.
$$ 
It is worth to mention that, in such coordinates, an equidistant
hypersurface with one end at $p_\infty$ is represented
by 
\[
\rho = \log |x- x_0| + C
\]
for the other end at some point $x_0 \in \mathbb{R}^n\subset\partial_\infty\mathbb{H}^{n+1}$ 
and a constant $C$.
\\

For a vertical graph $\rho=\rho(x)$ in Busemann coordinates,
one considers the inscribed radially symmetric graph (which is called
inner rotation hypersurface in \cite{AlCu,AlCu2}). More precisely,
let 
\[
\hat{\rho}(r)=\sup_{|x|=r}\rho(x).
\]
To see the use of inscribed radially symmetric graphs $\hat\rho$, 
similar to what was observed when hypersurfaces were assumed  
to be nonnegatively curved in \cite{AlCu,AlCu2}, we first observe:

\begin{lemma}\label{Lem:under equidistant} Suppose that the graph
$\rho=\rho(x)$ over $\Omega\subset\mathbb{R}^{n}$ in Busemann coordinates in hyperbolic space 
is complete and with nonnegative Ricci and one single end at $p_\infty$. 
Then $\Omega = \r^n$ and there is an equidistant hypersurface $\rho=\log|x|+C$ such that
\begin{equation}\label{Equ:rho upper bound}
\rho (x)\leq\hat \rho (|x|) \leq \log|x|+C
\end{equation}
for all $|x|$ sufficiently large. 
\end{lemma}

\begin{proof} First of all, we know the hypersurface is globally and strictly convex. 
Let $\hat{\Sigma}$ be the inscribed radially symmetric 
hypersurface as the graph of $\hat\rho$ to the hypersurface as the graph of $\rho$. 
It is easy to see that $\partial_{\infty}\hat{\Sigma}=\{p_\infty\}$. 
\\

Apply maximum principle to $\rho$, we can prove that $\hat\rho$ is non-decreasing. From the fact that
 $A\log r+B$ is $n$-harmonic and maximum principle, following the argument of \cite[Page 66]{HK}, we can show that
$\hat\rho$ is convex in $\log r$. Hence
$\hat\rho$ is continuous and differentiable except at countably many points.
Moreover, at a singular point $a$, $\hat\rho'_{-}(a) < \hat\rho'_{+}(a).$ 
When $\hat\rho$ is differentiable for $r\in (a, b)$, the corresponding
portion of $\hat{\Sigma}$ has nonnegative Ricci curvature. Because,
for any fixed $r \in(a,b)$, $\hat\Sigma$ is supported by $\Sigma$ at
least at some point $x$ with $|x|=r$. By the comparison of principal curvatures, one
may easily derive that Ricci of $\hat\Sigma$ is nonnegative from that the Ricci of $\Sigma$
is nonnegative. Therefore $\hat\Sigma$ is with Ricci curvature nonnegative everywhere 
on the regular part of $\hat\Sigma$. 
\\

Now, assume without loss of any generality that $0\in \Omega$. Let $R$ be the radius of the maximal ball 
$B(0, R)\subset\Omega$. For any fixed $r_0< R$, we take $C$ sufficiently large such that
$$
\hat \rho (r_0) <  \log r_0 + C.
$$
Here $\rho = \log |x| + C$ is the equidistant hypersurface about the vertical geodesic line $\gamma$ 
connecting $p_\infty$ and $0\in \mathbb{R}^n\subset\partial_\infty\mathbb{H}^{n+1}$. Then we claim 
\begin{equation}\label{Equ:rho upper}
\hat\rho (r) \leq \log r + C
\end{equation}
for all $r\in (r_0, R)$. Assume otherwise, there is some interval $[r_1,r_2]\subset(r_0, R)$,
such that 
$$
\hat\rho(r_1)=\log r_1+C \text{ and } \hat\rho (r)> \log r+C \text{ for $r\in(r_{1},r_{2}]$}.
$$
Then there has to be some $\xi\in(r_1, r_2)$ where $\hat\rho$ is differentiable and $\hat\rho'(\xi)> 1 / \xi$. 
This implies, the horizontal spherical section of $\hat\Sigma$ at $r=\xi$ is with negative definite second fundamental form,  
in contrast to the equidistant hypersurface, whose horizontal spherical sections are totally geodesic. 
Because that $\hat{\Sigma}$ is with nonnegatively Ricci when it is differentiable, and that the mean curvature of the spherical
section only drops at singular point, one may derive that $\hat{\Sigma}$ can only be compact, which clearly is a contradiction.
So we proved \eqref{Equ:rho upper}. To see $R=\infty$, we assume otherwise. Then, from the fact that $\Sigma$ is complete
and has only one end at $p_\infty$, $\lim_{r\to R}\hat\rho (r) = \infty$, which contrdicts with 
\eqref{Equ:rho upper}. Hence $\Omega = \r^n$ and \eqref{Equ:rho upper}
holds for all $r\geq r_0$. Thus the proof of the lemma is completed.
\end{proof}

Based on Theorem \ref{Thm:AH-1} and Theorem \ref{Thm:main theorem-2}, we are able to
improve the results on asymptotic behavior of global vertical graph of nonnegative sectional curvature 
in \cite{AlCu, AlCu2}. Namely,  

\begin{theorem} \label{Thm:main theorem-3} Suppose that $\Sigma$
is a properly embedded, complete hypersurface with nonnegative Ricci
and single end. Then it is a global graph of $\rho= \rho(x)$ in Busemann coordinates
and it is asymptotically rotationally symmetric in the sense that
there is a number $m\in[0,1]$ such that 
\[
m\log|x|+o(\log|x|)\leq \rho(x)\leq m\log|x|+C
\]
as $x \to \infty$ in $\r^n$. Moreover, $m=0$ implies that the hypersurface is a horosphere.  In any case, the
hypersurface $\Sigma$ always stays inside a horosphere and is supported by some equidistant hypersurface. 
\end{theorem}

\begin{proof}
As the first step, to use Theorem \ref{Thm:AH-1}, we first want to change a coordinatees in hyperbolic space, 
that is, to choose a different point at infinity $\partial_\infty\mathbb{H}^{n+1}$ for the half space model. Then, 
based on Lemma \ref{Lem:under equidistant}, we know
a priori that the hypersurface $\Sigma$ is below an equidistant hypersurface at least near the end at $p_\infty$. Hence,
in the new Busemann coordinates, the hypersurface $\Sigma$ is no longer a global graph of the height function, rather, a graph of the height 
function over a punctured ball, $B(0,R)\backslash\{0\}$ in the new Busemann coordinates$(y, \tau)$
(in other words, we may put the end at infinity of $\Sigma$ at the origin of the new Busemann
coordinates). Therefore we are looking at the part of the 
hypersurface $\Sigma$ that is parametrized as the graph of the height function $\tau = \tau (y)$, which is a 
$n$-subharmonic function in $B(0, R)\setminus\{0\}$ with $\lim_{y\to 0}\tau (y) = -\infty$. 
Thus, we may apply Theorem \ref{Thm:AH-1} to $-\tau$ and obtain
\begin{equation}\label{Equ:tau}
\aligned
\tau (y) \leq - m_1 \log\frac 1{|y|} + C  & \text{ for all $y\in B(0, R)\setminus\{0\}$} \\
\tau (y) \geq - m_1 \log\frac 1{|y|} + o(\log\frac 1{|y|}) & \text{ for all $y \in (B(0, R)\setminus\{0\})\setminus  E$} 
\endaligned
\end{equation}
for some $m_1 \geq 1$ and a set $E$ that is $n$-thin at the origin. \eqref{Equ:tau} can be improved to
\begin{equation}\label{Equ:tau improved}
- m_1 \log\frac 1{|y|} + o(\log\frac 1{|y|}) \leq \tau (y) \leq - m_1 \log\frac 1{|y|} + C   \text{ for all $y\in B(0, R)\setminus\{0\}$}
\end{equation} Let us assume this temporarily.

For the convenience of readers, we have here 
the transformation laws of the change of parameterizations of hyperbolic space from ball model to half space model (cf. \cite[Chapter 4]{Rat}):
$$
\left\{\aligned
y & = \frac {2z}{|Z - e_{n+1}|^2}\\
y_{n+1} & = \frac {1 - |Z|^2}{|Z - e_{n+1}|^2}
\endaligned\right.
\quad \text{ and } \quad
\left\{\aligned
x & = \frac {2z}{|Z + e_{n+1}|^2}\\
x_{n+1} & = \frac {1 - |Z|^2}{|Z + e_{n+1}|^2}
\endaligned\right.
$$
for $Z=(z, z_{n+1}) \in B(0, 1)\subset\r^{n+1}$, $Y = (y, y+{n+1})\in \r^{n+1}_+$ and $X = (x, x_{n+1})\in \r^{n+1}_+$, 
where $e_{n+1} = (0, 1)$ is the north pole of the unit sphere in $\r^{n+1}$. In $Y$ coordinates it takes the north pole 
to infinity and the south pole to the origin; while in $X$ coordinates it takes the south pole to infinity and the north 
pole to the origin. Hence the coordinate change between $X$ and $Y$ is the inversion with respect to the unit sphere centered at 
the origin:
$$
Y = \frac X{|X|^2} \quad \text{ or } \quad \left\{\aligned y & = \frac {x}{|x|^2 + x_{n+1}^2}\\
                                                              y_{n+1} & = \frac {x_{n+1}}{|x|^2 + x_{n+1}^2}.
                                                              \endaligned\right.
$$ 
We may assume from the beginning that the hypersurface $\Sigma$ has its end at $p_\infty = -e_{n+1}$. Therefore
$$
\frac 1{|y|}  = |x|  \cdot (1 + \frac {x_{n+1}^2}{|x|^2}) \text{ and } 
\tau  =  \log y_{n+1} = \rho - 2 \log |x| - \log (1 + \frac {x_{n+1}^2}{|x|^2}).
$$
So we may translate \eqref{Equ:tau improved} into
\begin{equation}
m \log |x|  + o(\log |x|)\leq\rho(x) \leq m \log |x|  + C   \text{ for all $x \in \r^n$}\label{Equ:rho-theorem 3.1-1} 
\end{equation}
for  some $m = 2- m_1 \leq 1$. Here we use \eqref{Equ:rho upper bound} 
from Lemma \ref{Lem:under equidistant} to control $x_{n+1}^2 / |x|^2$.
\\

Next step is to improve \eqref{Equ:tau} and eliminate any nontrivial $n$-thin set $E$. Our approach here is  to use the 
strict and global convexity of the hypersurface $\Sigma$ to rule out the nontrivial $n$-thin set $E$, which is close to that
in \cite{AlCu, AlCu2} in 2 dimensions but more straightforward. 
Assume otherwise, \eqref{Equ:tau improved} is not true on a set $E$, which is $n$-thin and non-compact. 
Hence, there is a positive number $\epsilon_0$ and a sequence point $p_k = (s_k\theta_k, \tau(s_k\theta_k))\in\Sigma$ such that
\begin{equation}\label{Equ:height at E}
y_{n+1}^\Sigma (s_k\theta_k)= e^{\tau (s_k\theta_k)} < s_k^{m_1 + \epsilon_0} 
\end{equation}
and $s_k\to 0$.  We have, in the light of Lemma \ref{Lem:gehring} and Definition \ref{Def:n-thin},
for each $s_k\theta_k\in E$, there always exists $\hat s_k \theta_k\notin E$ for $\hat s_k \in (s_k(1 - s_k^l), s_k)$ for any fixed large 
$l \geq 1$. This can be proved by contradiction. Assume otherwise, one derives from \eqref{Equ:precise gehring estimate} that, 
for each $i$, 
\begin{equation}
cap_n(E\cap\omega_i, \Omega_i) \geq \frac {c_n}{( 2i(l+1)\log 2)^{n-1}},
\end{equation}
by the scaling invariance, which is impossible by Definition \ref{Def:n-thin}. On the other hand, there is $\delta_0$ such that
$$
y_{n+1}^\Sigma (s\theta) = e^{\tau(s\theta)} \geq s^{m_1 + \frac 12 \epsilon_0}
$$
for all $s\theta \notin E$ and $0 < s < \delta_0$. In particular
\begin{equation}\label{Equ:height at left and right}
y_{n+1}^\Sigma(\hat s_k \theta_k) \geq \hat s_k^{m_1 + \frac 12\epsilon_0}  \geq a_0 s_k^{m_1 +\frac 12\epsilon_0}
\end{equation}
for some positive $a_0$, at least when $k$ is large. Let us assume the following is the equation for the semi-circle that is inside the 
hyperplane tangent to $\Sigma$ at the point over $\hat s_k\theta_k$ and in the 2-plane for the fixed $\theta_k\in \mathbb{S}^{n-1}$
$$
|s - c_k|^2 + y_{n+1}^2 = r_k^2 = |\hat s_k - c_k|^2 + (y^\Sigma_{n+1}(\hat s_k\theta_k))^2
$$
where $(c_k, 0)$ is the center of the semi-circle and $0 < c_k < s_k$ due to the fact that 
$$y^\Sigma_{n+1}(\hat s_k\theta_k) < y^\Sigma_{n+1}(s_k\theta_k).$$ 
We may estimate the height of this semi-circle at $s = s_k$:
\begin{align}
y_{n+1}^2|_{s = s_k} & \geq  a_0^2 s_k^{2m_1 +\epsilon_0} + |\hat s_k - c_k|^2 - |s_k - c_k|^2 \nonumber\\
& = a_0^2 s_k^{2m_1 +\epsilon_0} + |\hat s_k - s_k|^2  - 2 (s_k- \hat s_k) \cdot (s_k- c_k)  \nonumber \\
&\geq a_0^2 s_k^{2m_1 +\epsilon_0} - c_0 s_k^{2m_1 + 1 + 2\epsilon_0}  > (y^\Sigma_{n+1} (s_k\theta_k))^2
\label{Equ:height on the hyperplane} 
\end{align}
for some uniform $a_0$ and $c_0$ and some appropriately large $l$, in the light of \eqref{Equ:height at E}, 
which means the point $p_k$ on $\Sigma$ falls under the hyperplane and violates the strict and global convexity of 
$\Sigma$, at least when $k$ is large enough.
\\

So far we have shown that 
$$
m \log |x| + o(\log |x|) \leq \rho (x) \leq m \log |x| + C
$$
as $x \to \infty$ in $\r^n$ with $m \leq 1$. In the last step, we prove that $m\in [0, 1]$ and $\Sigma$ is a horosphere when $m=0$.
We at this point go back to the Busemann coordinates $(x, \rho)$, use the similar argument in the last step of the proof of 
Theorem \ref{Thm:main theorem-g} (even easier, because that there is no bad thin set), and obtain
$$
|m|^{n-2}m = \frac 1{w_{n-1}} \int_{\r^n} (\Delta_n \rho)dx \geq 0.
$$
Therefore, when $m=0$, $\rho$ in fact is an $n$-harmonic function and upper bounded by a constant.  In the light of Liouville Theorem in
\cite[Theorem 6.2 and Corollary 6.11]{HKM}, $\rho$ is a constant, i.e. $\Sigma$ is a horosphere.
\\

At last, it is easily seen that $\Sigma$ stays inside a horosphere, when $m>0$ or $m=0$, that is, there is some constant $C$ 
such that 
\[
\rho (x_{1},\cdots,x_{n})\geq C.
\]
The fact that $\Sigma$ is supported by some equidistant hypersurface is proved in Lemma \ref{Lem:under equidistant}.
The proof of Theorem \ref{Thm:main theorem-3} is complete.
\end{proof}

{\bfseries Acknowledgement} The first named author would like to thank Igor Verbitsky for helpful discussions. 

%%%%%%%%%%%%%%%%%%%%%%%%%%%%%%%%%%%%%%%%%%
%%%%%%%%%%%%%%%%%%%%%%%%%%%%%%%%%%%%%%%%%%

\end{document}